\documentclass[12pt,reqno,twoside,a4paper]{amsart}

\makeatletter
\@namedef{subjclassname@2020}{\textup{2020} Mathematics Subject Classification}
\makeatother

\usepackage{amsfonts,amsthm,amssymb,amsmath,amscd,graphicx}
\usepackage{color}

 at 10truept
 at 10truept
 at 9truept
\font \bigrm=cmr10 at 21truept

\newtheorem{proposition}{Proposition}[section]
\newtheorem{theorem}[proposition]{Theorem}
\newtheorem{corollary}[proposition]{Corollary}
\newtheorem{lemma}[proposition]{Lemma}

\theoremstyle{definition}
\newtheorem{definition}[proposition]{Definition}
\newtheorem{example}[proposition]{Example}

\newtheorem{free text}[proposition]{}

\newcommand{\fg}{\mathfrak {g}}
\newcommand{\fgl} {\mathfrak {gl}}

\newcommand{\ft}{\mathfrak {t}}

\newcommand{\fl}{\mathfrak {l}}

\newcommand{\fu}{\mathfrak {u}}

\newcommand{\gz}{\mathfrak {g}_{\bar{0}}}
\newcommand{\go}{\mathfrak {g}_{\bar{1}}}

\newcommand{\ch}{\mathcal H}
\newcommand{\Lie}{\mbox{\rm Lie}}

\newcommand{\sk}{\bigskip}

\newcommand{\lra}{\longrightarrow}

\newcommand{\pa}{\partial}

   \newcommand{\cc}{{\mathcal C}}

\newcommand{\de}{\delta}
\newcommand{\ep}{\varepsilon}

\newcommand{\la}{\lambda}

\newcommand{\om}{\omega}

\newcommand{\sh}{\sharp}

\newcommand{\im}{{\rm Im}}

\newcommand{\bl}{\mathbb {L}}
\newcommand{\bz}{\mathbb {Z}}
\newcommand{\bc}{\mathbb {C}}

\newcommand{\bk}{\mathbb {K}}

\newcommand{\br}{\mathbb {R}}

\newcommand{\N}{{\mathbb {N}}}
\newcommand{\C}{{\mathbb {C}}}
\newcommand{\R}{{\mathbb {R}}}

\newcommand{\cD} {\mathcal{D}}

\newcommand{\cH} {\mathcal{H}}

\newcommand{\cO} {\mathcal{O}}

\newcommand{\id}{{\rm id}}

\newcommand{\wh}{\widehat}

\newcommand{\blz}{\bl_{\bar{0}}}

\newcommand{\Om}{\Omega}

\newcommand{\we}{\wedge}

\newcommand{\beq}{\begin{equation}}
\newcommand{\eeq}{\end{equation}}

\newcommand{\zero}{{\bar{0}}}
\newcommand{\one}{{\bar{1}}}

\newcommand{\ze}{\zeta}

\newcommand{\bere}{\mathcal{B}}
\newcommand{\F}{2F}

\numberwithin{equation}{section}

\theoremstyle{plain}

\begin{document}

%%%%%%%%% TITLE  AND OTHERS %%%%%%%%%%%%%

{\ }
 \vskip-43pt

   \centerline{\bigrm SUPER K\"AHLER STRUCTURES}
\vskip15pt
   \centerline{\bigrm ON THE COMPLEX ABELIAN}
\vskip15pt
   \centerline{\bigrm LIE SUPERGROUPS}

\vskip25pt

   \centerline{Meng-Kiat CHUAH$\,{}^\flat$ \ ,  \ \ Fabio GAVARINI$\,{}^\sharp$}

\vskip21pt

   \centerline{$ {}^\flat $  Department of Mathematics, National Tsing Hua University}
   \centerline{Hsinchu 300, TAIWAN   ---   {\tt chuah@math.nthu.edu.tw}}
 \vskip9pt
   \centerline{$ {}^\sharp $  Dipartimento di Matematica, Universit\`a degli Studi di Roma ``Tor Vergata''}
   \centerline{Via della ricerca scientifica 1,  I\,-00133 Roma, ITALY   ---   {\tt gavarini@mat.uniroma2.it}}

\vskip39pt

   {\bf Abstract:}  Let $G$ be a real Abelian Lie supergroup,  let $M$ be its complexification.
   We classify the $G$-invariant super K\"ahler forms on $M$.
   For the super K\"ahler forms with Hamiltonian actions,
   we extend the scheme of geometric quantization
   to the super setting and construct unitary $G$-representations.
We show that the irreducible representations that occur in these unitary representations
are governed by the image of the moment maps of super K\"ahler forms.
As an application, we construct a Gelfand model of $G$,
namely a unitary $G$-representation in which every unitary irreducible representation
occurs exactly once.

\vskip15pt

\noindent
 \emph{2020 Mathematics Subject Classification:}\;  22E45, 32A36, 53C55, 53D50, 58A50
 \vskip3pt

   \emph{Keywords:} Lie supergroups, super K\"ahler forms, geometric quantization, unitary representations, moment maps.

{\ } \vskip-43pt

\tableofcontents

% \newpage
\bigskip
   \bigskip

  %%%%%%%%%%%%%%%%%%%%%%%%%%%
%%%%%%%%%   PAPER   %%%%%%%%%%%

%%%%%%%%%%%% INTRODUCTION %%%%%%

\section{Introduction}
  \label{sec: intro}

\setcounter{equation}{0}

Let $G$ be a real Lie group.
The theory of geometric quantization \cite{ko} associates the $G$-action on a symplectic manifold $M$ to a unitary $G$-representation $\cH$, and one studies its irreducible subrepresentations.  It is natural to consider this construction in the super setting.  This is a challenging task, as non-trivial unitary representations are absent in many cases of real simple Lie supergroups \cite[Thm.6.2.1]{ns}  and Heisenberg-Clifford Lie supergroups  \cite[Thm.5.2.1]{sa}.  In this article, we consider the Abelian Lie supergroups, a case somehow ``transversal'' to those in  \cite{ns}  and  \cite{sa}.

   The geometric quantization of the actions of connected Abelian Lie groups on their complexifications
has been carried out successfully in  \cite{jga,pems}.
We now consider its super analogue.
Let $G$ be a connected real Abelian Lie supergroup.
Its even part is
$ G_\zero \cong T_n \times \br^m $, where $T_n$ is the $n$--dimensional real torus.
 As for any Lie supergroup  (cf.\ \cite{Ga1,Ga2}),  we have a global splitting of  $ G $  of the form
\beq  G  \, = \,  T_n \times \br^m  \times {\textstyle \bigwedge}_k^\br \,,
\label{abeli}
\eeq
where
  $ \, \bigwedge_k^\br \, $  is the supermanifold associated with  the real Grassmann  algebra  in  $ k $  odd indeterminates.
  The defining superalgebra of smooth functions on  $ G $  factors into
  $$  \cO_G(G_\zero)  \, = \,  C^\infty(G_\zero) \otimes \Lambda_\br(\xi_1,\dots,\xi_k) .$$
 In particular, the local structure around a single point in  $ G_\zero $  can be described by
  a local chart,  denoted by  $(x , \xi) $.

Let $\ft_n$ be the Lie algebra of $T_n\,$.
The Lie algebra of the additive group $\br^m$ is just $\br^m$ itself,
and its exponential map  $ \, \br^m \relbar\joinrel\lra \br^m \, $  is the identity map.
Let $ \, \la \in i (\ft_n + \br^m)^* \, $,
namely  $ \, \la : \ft_n \times \br^m \!\relbar\joinrel\lra i\,\br \, $.  We say that $\la$ is an integral weight if
it determines a character  $ \, \chi_\la : T_n \times \br^m \!\relbar\joinrel\lra S^1 \, $  such that the diagram commutes,
\beq
 \begin{array}{ccc}
\ft_n \times \br^m & \stackrel{\la}{\lra} & i\,\br \\
\downarrow & & \downarrow \\
T_n \times \br^m & \stackrel{\chi_\la}{\lra} & S^1
\end{array}  \label{niw}
\eeq
where the downward arrows are exponential maps.
We write  $  \la = \la_1 + \la_2  $
and $\chi_\la = \chi_{\la_1} \chi_{\la_2}$,
where  $  \la_1 \in i \, \ft^*  $ and  $ \, \la_2 \in i \, {( \br^m )}^*  $.

Let $ \fg = \gz + \go := \Lie(G) $ be the tangent Lie superalgebra of  $ G  $.
A unitary representation of $G$ is a super Hilbert space
with compatible actions by  $G_\zero$ and $\fg$
(see Definition \ref{abi} and (\ref{cmpt}), or
\cite[\S4]{fg}\cite[Def.4.2.1]{ns} for more details).
It is said to be irreducible if it has no proper subrepresentation.
                                                                \par

Let $\wh{G}$ denote the set of all irreducible unitary $G$-representations, up to equivalence;
we define $\wh{G_\zero}  $, $\wh{T_n}$ and $\wh{\br^m}$ similarly.
All the elements of $\,\wh{T_n}$ are 1-dimensional, and are parametrized by the integral weights $\la_1$, where $V_{\la_1} \in \wh{T_n}$
consists of vectors $v$
 such that
 $ \, t \cdot v = \chi_{\la_1}(t) v \, $
for all $\,t \in T_n\,$.
%
% We identify these integral weights with $\bz^n$ and write
%
 We identify these weights with $\bz^n$,  hence
\[ T_n \, = \, \br^n/\bz^n \;\;, \quad i \, \ft^* \, \cong \, \br^n \;\;,  \quad  \wh{T_n} \, \cong \, \bz^n \;\; .\]

There is no
obstruction to the existence of $\chi_{\la_2}\,$, so $\wh{\br^m} = \br^m$, and
$\la$ is integral if and only if $\la_1$ is integral.
The integral weights are identified with $\wh{G_\zero}\,$, so that
  \[ \wh{G_\zero}  \; \cong \;  \bz^n \times \br^m . \]

\vskip9pt

\begin{theorem}  \label{dual2}
Let  $ \, G   =   T_n \times \br^m \times {\textstyle \bigwedge}_k^\br  $.
%%%%%%%%%
  Every irreducible unitary representation of  $ \, G $  is 1--dimensional, hence its superdimension is either  $ \, 1|0 \, $  or  $ \, 0|1 \, $.
   Therefore,  $ \wh{G} $  has a natural group structure given by
  \[ \wh{G} \,\;\cong\;\, \wh{G_\zero} \times \bz_2 \,\;\cong\;\,
  \bz^n \times \br^m \times \bz_2 . \]
The representation space parametrized by  $ \, (\la, \ep) \in (\bz^n \times \br^m) \times \bz_2 \, $
consists of vectors $v$ such that
$ \, g \cdot v = \chi_\la(g) v \, $  for all  $ \, g \in  \, T_n \times \br^m $
and  $ \, \xi \cdot v = 0 \, $  for all  $ \, \xi \in \fg_\one \, $.
\end{theorem}

\sk

By Theorem \ref{dual2}, we write
  \[  \wh{G}  \, = \,  \{ V_\la^\pm \;;\; \la \in \bz^n \times \br^m \, , \, \pm \in \bz_2 \}  .  \]
We construct $V_\la^\pm$ explicitly in Example \ref{contoh}.

 Let $M$ be the complexification of $ G \, $:
it is the Lie supergroup with Lie superalgebra $\,\fg \otimes \bc\,$,
such that $M$ and $G$ have the same maximal compact subgroup. Therefore
\beq  \label{split2}
  M  \, = \,  M_\zero \times M_\one  \, = \,
  \bc^n \big/ i\,\bz^n \times \bc^m \times {\textstyle \bigwedge}_k^\bc \; .
\eeq
Here $M_\zero = \bc^n / i\,\bz^n \times \bc^m$, where
$\bc^n \big/ i\,\bz^n$ is the complexification of $T_n = \br^n \big/i \bz^n$,
and $\bc^m$ is the complexification of $\br^m$.
Also, $ \, M_\one := \bigwedge_k^\bc \, $  is described through
complex odd Grassmann variables  $ \ze_1 \, $,  $ \dots $,  $ \ze_k \, $.
We use the complex coordinates  $ z , \zeta $, where
\beq  \label{loc-coord}
\begin{array}{l}
M: \; (\, z  , \zeta \,) = (z_1 , \dots , z_{n+m} , \zeta_1 , \dots , \zeta_k) \, ;
\, z_r = x_r + i\,y_r \, \mbox{ and } \, \zeta_s = \xi_s + i\,\eta_s \, ,
\\
G: \; (\,  y  , \eta \,)   =
   (y_1 , \dots , y_{n+m} , \eta_1 , \dots , \eta_k) \; .
   \end{array}
   \eeq
So  $ G $  identifies with a real super subgroup of  $ M $,
and $G$ acts on $M$ on the imaginary part $(y,\eta)$.

  We define super K\"ahler forms and super symplectic forms in Definition \ref{dska}.
  In classical symplectic geometry with symmetry,
  Hamiltonian actions and moment maps play very important roles \cite[\S26]{gs}.
  They arise in symplectic manifolds
such as coadjoint orbits and cotangent bundles,
and carry applications in
representation theory and theoretical mechanics.
 We extend these notions to the super setting in Definition \ref{momm}.

 We identify a real $d \times d$ matrix $X =(x_{pq})$ with a bilinear form on $\br^d$ by
 \beq
  X(u,v) = \sum_{p,q} x_{pq} u_q v_p \label{fhua}
  \eeq
  for all $u=(u_q), v=(v_p) \in \br^d$.
If $X$ is symmetric and satisfies $X(u,u) > 0$ for all $u \neq 0$,
 we say that $X$ is positive definite.
 Given a function $F=F(x)$ on $\br^d$, we let $F'' = (\frac{\partial^2 F}{\partial x_p \partial x_q})$
 denote its Hessian matrix.
 We say that $F$ is strictly convex if $F''$ is positive definite.
 In what follows, the indices $p,q$ cover $1,...,n+m$ with even variables,
 and cover $1,...,k$ with odd variables.

\begin{theorem} \label{kah2}
$\,$

\noindent {\rm (a)} \,
A closed $G$-invariant real $(1,1)$-form on $M$ is given
by  $ \, \om = \om_\zero + \om' \, $,
\[ \om_\zero  \; = \;  \sum\nolimits_{p,q} \, \frac{\pa^2 F}{\pa x_p \, \pa x_q} \, dx_p \we dy_q
+ \sum\nolimits_{p<q} c_{pq} \, \big( dx_p \we dx_q + dy_p \we dy_q \big) \,, \]
\[ \om'  \; = \;  \sum\nolimits_{p,q} \, k_{pq} \, \big(\, d\xi_p \, d\xi_q + d\eta_p \, d\eta_q \,\big)
   \, + \, \sum\nolimits_q \, d h_q \, d \eta_q \,,   \]
 where $c_{pq} =-c_{qp} \in \br$,
$k_{pq}= k_{qp} \in \br$,
 $h_q = h_q(\xi)$ are Grassmann polynomials in $\xi$.

 \noindent {\rm (b)} \, Let $C = (c_{pq})$, $K=(k_{pq})$, $H = \big( \frac{\pa h_p}{\pa \xi_q} \big)$.
 Then $\om$ is super K\"ahler if and only if the symmetric matrices
$\left( \begin{array}{cc}
 F'' & -C \\
 C & F''
 \end{array} \right)$
 and
 $\left( \begin{array}{cc}
 2K & -H \\
 H & 2K
 \end{array} \right)$
 are positive definite everywhere.

  \noindent {\rm (c)} \, The $G$-action is Hamiltonian if and only if
   $C=K=0$. In this case its moment map is
  $$  \Phi : M \lra \fg^* \;\; ,
   \quad   \Phi(x,y,\xi,\eta) = \Big(\frac{\pa F}{\pa x_p} \,,\, - h_q(\xi) \Big) .  $$
\end{theorem}

\sk

We apply Theorem \ref{kah2} and fix a $G$-invariant super K\"ahler form $\om$ on $M$,
where the $G$-action is Hamiltonian with moment map $\Phi$.
We extend the standard machinery of geometric quantization \cite{ko}  to the super setting and obtain a holomorphic
Hermitian line bundle $\bl$ on $M$.
Let $\cH(\bl)$ denote its holomorphic sections.
We define the star operator $f \mapsto f^*$ on $C^\infty(\bl)$
  --- see (\ref{starr})  ---   then apply Berezin integration \cite{va}
to construct the super Hilbert space
%(see Definition \ref{suhi})
%
  \beq
    \cH^2(\bl)  \,\; := \;\,  \big\{\, f \in \cH(\bl) \;;\;
\int_M \! f f^* \, d\mathcal{B} \;\; \ \mbox{converges} \,\big\} \;\; .  \label{inhe}
\eeq
                                                     \par
  The  $ G $-representation  on $\cH^2(\bl)$ is not unitary.
  It contains $\cH^2(\blz)$ as the largest $G$-subrepresentation
in which the $G$--action is unitary, and we
consider the irreducible unitary subrepresentations which occur in  $\cH^2(\blz)\,$.

  Recall that,  by Theorem \ref{dual2},
  $ \wh{G} = \big\{\, V_\la^\pm \,\big|\, \lambda \in \bz^n \times \br^m \,\big\} $.
For the factor $\br^m$,
its Plancherel measure provides zero measure on each member
(unlike $\bz^n$, whose members have point mass).
For this reason, the occurrence of a subrepresentation is understood
as appearance in the direct integral decomposition of $\cH^2(\blz)$,
see Definition \ref{weakocc}.  With this in mind, we have the following theorem.
Let $\im(\Phi)_\zero \subset \gz^*$ denote the even part of the image of the moment map $ \Phi $.

\sk

\begin{theorem}  \label{occ2}
Let  $ G   =   T_n \times \br^m  \times {\textstyle \bigwedge}_k^\br $.  Then
 $ \cH^2(\bl)  $  is a super Hilbert space, and  $ \cH^2(\bl_\zero)  $
  is its largest  $ G $-subrepresentation  in which the  $ G $-action is unitary.
   Moreover,  $  \cH^2(\bl_\zero)  $  is multiplicity free,
   with  $  V_\la^+ $  occurring if and only if  $ \la \in {\im(\Phi)}_\zero $.
   Also,  $  V_\la^- $  does not occur in  $  \cH^2(\bl_\zero) $,  for any
 $ \lambda  $.
\end{theorem}

\sk

 According to Gelfand, a  \textit{model}  of a Lie group is a unitary representation on a Hilbert space
in which every irreducible representation occurs exactly once \cite{gz}.
The model of $G_\zero$ has been constructed in \cite[Cor.3.3]{pems}.
It is natural to extend this notion to the super setting,
so we say that a model of $G$ is a unitary representation
on a super Hilbert space in which every member of $\wh{G}$ occurs once.
   \hbox{We now construct a model.}

By Theorem \ref{occ2}, the odd representations
$V_\la^-$ do not occur in $\cH^2(\bl_\zero)\,$.
To remedy this defect, we apply the involutive endofunctor
 $\Pi$ that is the identity on each object  but reverses the
 $ \bz_2 $-grading,
 namely
\beq  \label{switch}
  \Pi \, V_\la^+  \, = \;  V_\la^- .
\eeq

We now apply Theorems \ref{kah2} and \ref{occ2} to construct a model of $G$ as follows.
Let $F \in C^\infty(\br^{n+m})$ be a strictly convex function, and all $c_{jk} = k_{pq} = h_p =0$,
so that $\om$ as appears in Theorem \ref{kah2} is a super K\"ahler form on $M$
with Hamiltonian $G$-action. Let
   $F' = \big( \frac{\pa F}{\pa x_j} \big) : \br^{n+m} \lra \br^{n+m} = \fg^*$
   denote the gradient mapping.

\sk

\begin{corollary}  \label{mode}
 Suppose that  $ \, F' \, $  is surjective.  Then
 $ \ch^2(\bl_\zero) \oplus \Pi \, \ch^2(\bl_\zero) $  is a model of  $  G  $.
\end{corollary}

\sk

   In view of  Theorem \ref{occ2},  one might wonder if $ \ch(\bl)$ contains
   any $G$-subrepresentation beyond $ \ch(\bl_\zero)$ which is irreducible or unitarizable (apart from
    using the $L^2$-structure (\ref{inhe})).  In this respect, we find the following answer, in the negative:

\sk

\begin{theorem}  \label{thm: L-NO-unitary}
Every irreducible or unitarizable $G$-subrepresentation of $\ch(\bl)$
is contained in  $\ch(\bl_\zero)  $.
\end{theorem}

\sk

   We organize the sections of this article as follows.
Section \ref{sec: supergeometry} recalls the notions and language of Lie superalgebras and Lie supergroups.
Section \ref{irredu} proves Theorem \ref{dual2}, which classifies the irreducible unitary  representations
$\wh{G}$.
Section \ref{superkahler} proves Theorem \ref{kah2}, which classifies
the $G$-invariant super K\"ahler forms on $M$ and studies their moment maps.
Section \ref{geomquan} performs geometric quantization and proves Theorem \ref{occ2},
which studies the irreducible representations that occur
in $\ch^2(\bl )$. It also proves Corollary \ref{mode},
which constructs a model of $G$ in terms of $\ch^2(\bl_\zero)$.
Section \ref{nono} proves Theorem \ref{thm: L-NO-unitary},
which restricts the irreducibility and unitarizability of
subrepresentations of $ \, \ch \big(\bl \big)$.

% \newpage
\bigskip
   \bigskip

 %%%%%%%%%%%%%%%%%%%%%%%%%%%%%%%%%%%%%%%%%%%%%%%%%%%%%%%%%%%%%%%%%
%%%%%%%%%%%%%%%%%%%%%%%%%%%%%%%%%%%%%%%%%%%%%%%%%%%%%%%%%%%%%%%%%%
%%  SECTION TWO
%%%%%%%%%%%%%%%%%%%%%%%%%%%%%%%%%%%%%%%%%%%%%%%%%%%%%%%%%%%%%%%%%%
 %%%%%%%%%%%%%%%%%%%%%%%%%%%%%%%%%%%%%%%%%%%%%%%%%%%%%%%%%%%%%%%%%%

\section{Preliminaries in supergeometry}  \label{sec: supergeometry}
\setcounter{equation}{0}

 In this section, we recall the notions and languages of Lie superalgebras and Lie supergroups.
 We shall always assume that their base fields are $\br$ or $\bc$.
 Everything indeed is standard matter, we just fix the terminology.

\medskip

\subsection{Basic superobjects} \label{subsec:basic-sobjcs}
 Let $ \bz_2 = \{\zero,\one\} $ be the 2-element group.
 A  \textit{super vector space} is a
 $ \bz_2 $-graded vector space $ V = V_\zero \oplus V_\one $.
 Then $ V_\zero $ and its elements are called  even,  $ V_\one $ and its elements
  are called odd. By  $ |x| \in \bz_2 $  we denote the parity of a homogeneous element $x$.

   We call  \textit{superalgebra}  any associative, unital  $ \bz_2 $-graded
   algebra $ A = A_\zero \oplus A_\one $, where  $ A_{\overline{\text{a}}} A_{\overline{\text{b}}} \subseteq A_{\overline{\text{a}}+\overline{\text{b}}} $.
   It is said to be  \textit{commutative}  if  $  x y = (-1)^{|x| |y|} y x  $  for all homogeneous  $ x, y \in A $;  so in particular
   $ z^2 = 0 $  for all  $ z \in A_\one $.
   %All  $ \bk $--superalgebras  form a category, whose morphisms are those of unital  $ \bk $--algebras  preserving the  $ \bz_2 $--grading:
%%%%%
 %we write  $ \, \alg \, $  for this category;
%%%%%
 %inside it, commutative  $ \bk $--superalgebras  form a subcategory, denoted by  $ \, \salg \, $.
%%%%%
% We write  $ \, \alg \, $  for the category of (associative, unital) commutative  $ \bk $--algebras,  and
%%%%%Also, we write
%%%%%$ \, {(\text{\bf mod})} \, $  for the category of  $ \bk $--modules.  There exists an obvious functor  $ \, {(\ )}_\zero \colon \salg \longrightarrow \alg \, $ given on objects by  $ \, A \mapsto A_\zero \, $.
%

   A  \textit{Lie superalgebra}  is a super vector space $\fg = \fg_\zero \oplus \fg_\one $  with a Lie super bracket
    \[ [ \cdot, \cdot ] \colon \fg \times \fg \longrightarrow \fg \;\;,\;\; (x,y) \mapsto [x,y] ,\]
      which is bilinear,  preserves the $ \bz_2 $-grading, and  for all homogenenous $x, y, z \in \fg $,
\begin{enumerate}  \itemsep=0pt
  \item[(a)]  \quad  $ [x,y] + {(-1)}^{|x| |y|}[y,x] = 0 $  \quad  {\it $($anti-symmetry$)$},
  \item[(b)]  ${(-1)}^{|x| |z|} [x, [y,z]] + {(-1)}^{|y| |x|} [y, [z,x]] + {(-1)}^{|z| |y|} [z, [x,y]] = 0 $ {\it $($Jacobi identity$)$}.
\end{enumerate}
%All Lie $ \bk $--superalgebras form a category, denoted by $ \sLie_\bk \, $, whose morphisms are $ \bk $--linear, preserving the $ \bz_2 $--grading and the bracket. Note that $ \fg_\zero $ is automatically a Lie $ \bk $--algebra.

\smallskip

\subsection{Supermanifolds and supergroups}  \label{subsec: super-manfds/groups}
 We now briefly recall the notions of supermanifolds and Lie supergroups: see  \cite{ccf, va}  for more details.
 \vskip5pt
   \textsl{$ \underline{\text{Superspaces}} $.}  A  superspace  is a topological space  $S$  with a sheaf of commutative superalgebras  $ \, \cO_S \, $ such that the stalk  $ \, \cO_{S,x} \, $  of  $ \, \cO_S \, $  at each  $ \, x \in S \, $  is a local superalgebra.  A  morphism  $ \, \phi \colon S \longrightarrow T \, $  is a continuous map
with a  morphism of sheaves
$ \, \phi^* \colon \cO_T \longrightarrow \phi_*(\cO_S) \, $ such that  $ \, \phi_x^*(\mathfrak{m}_{\phi(x)}) \subseteq \mathfrak{m}_x \, $,  \,where  $ \, \mathfrak{m}_{\phi(x)} \, $  and  $ \, \mathfrak{m}_{x} \, $  are the unique maximal ideals in the stalks  $ \, \cO_{T,\phi(x)} \, $  and  $ \, \cO_{S,x} \, $ respectively.

   The holomorphic linear supervariety $\cH_\bc^{p|q} $  is the topological space $ \bc^p $ endowed with the sheaf of
   commutative superalgebras $  \cO_{\cH_\bc^{p|q}}(U) := \cH_{\bc^p}(U) \otimes_\bc \Lambda_\bc(\xi_1,\dots,\xi_q) =: \cH_{\bc^{p|q}}(U) $  for any open set $ U \subseteq \bc^p $, where $ \cH_{\bc^p} $ is the sheaf of holomorphic functions on $ \bc^p $, and
   $ \Lambda_\bc(\xi_1,\dots,\xi_q)  $ is the complex Grassmann algebra on odd variables $ \xi_1 ,..., \xi_q $.  A holomorphic supermanifold of super dimension $ p|q $ is a superspace $ \, M \, $ which is locally isomorphic to $ \cH_\bc^{ p|q} \, $  as a locally-ringed space, i.e. for each $ x \in M  $, there is an open set $ V_x \subseteq M $ with $ x \in V  $ and $  U \subseteq \bc^p  $
   such that $  \cO_M{\Big|}_V \cong \cO_{\cH_\bc^{p|q}}{\Big|}_U $.
   %A morphism between holomorphic supermanifolds is a morphism between them as complex superspaces.  We denote the category of holomorphic supermanifolds by ($\cat{hsmfd}$)\,.

   With a similar construction, we define
    real smooth supermanifolds. Here we replace $\cH_\bc^{ p|q} $  by $ \br^p $ endowed with  $  \cO_{\cc_\br^{p|q}}(U) := C^\infty_{\br^p}(U) \otimes_\br \Lambda_\br(\xi_1,\dots,\xi_q) =: C_{\br^{p|q}}^\infty(U) $, where  $ C^\infty_{\br^p} $
    is the sheaf of  smooth  functions on $ \br^p $,
    and we replace $\Lambda_\bc(\xi_1,\dots,\xi_q)$ by $ \Lambda_\br(\xi_1,\dots,\xi_q) $.
 \vskip5pt
   Given a real smooth supermanifold  $ M $  and an open subset  $ U $,  let us choose a local chart around a point in  $ U $:  the even coordinates  $ x_i $  in this chart along with $ \xi_j $  provide a  smooth local chart  for  $ U $   as a superspace, which we will later denote by  $  (x, \xi) := (x_1 , \dots , x_p , \xi_1 , \dots , \xi_q) $.
   For holomorphic supermanifold, one defines the holomorphic local chart
   in a similar manner.  See  \cite{ccf}  for further details.
%

%
   %Let now $ M $ be a holomorphic supermanifold and $ U $ an open subset in $M$. Let $ \ci_M(U) $ be the ideal of $ \cO_M(U) $ generated by the odd part of the latter: then $ \cO_M \big/ \ci_M $ defines a sheaf of purely even superalgebras over $ |M| \, $, locally isomorphic to  $ \, \cH_{\bc^p} \, $. Then the reduced manifold $ \, M_{\rm rd} := \big(\, |M| \, , \cO_M / \ci_M \,\big) \, $  is a classical  holomorphic manifold, called the  underlying holomorphic submanifold  of  $ M \, $. The projection $ \, s \mapsto \tilde{s} := s + \ci_M(U) \, $,  for  $ \, s \in \cO_M(U) \, $,  at the sheaf level yields an embedding $ \, M_{\rm rd} \longrightarrow M \, $, so  $ M_{\rm rd} $  can be seen as an embedded sub-supermanifold of  $ M \, $. The whole construction is functorial in  $ M \, $.

   %The same construction applies to real smooth supermanifolds as well.
%

%
   Each classical manifold
 can be seen as a supermanifold, just by regarding its structure sheaf as one of superalgebras that have trivial odd part.
 \vskip5pt
   \textsl{$ \underline{\text{Lie supergroups}} $.}  Any group object in the category of holomorphic supermanifolds
    (resp. real supermanifolds)
    is called a  holomorphic Lie supergroup (resp. real Lie supergroup).
    %These objects, together with their obvious morphisms, form a subcategory among supermanifolds, denoted by  $ {(\cat{Lsgrp})}_\bc \, $.  Similarly, the group objects in the category  ($\cat{ssmfd}$)  are called real smooth Lie supergroups:  together with their obvious morphisms, they form a subcategory  $ {(\cat{Lsgrp})}_\br $  of  ($\cat{ssmfd}$)\,.
 The functor Lie links Lie supergroups $G$ to Lie superalgebras
 $\fg = \mbox{Lie}(G)$.  The super version of the correspondence between Lie groups and Lie algebras hold true.
%

%
%    A relevant aspect of the theory of holomorphic or smooth supermanifolds is that they can
% be entirely studied in terms of the algebra of global sections of their structure sheaf, i.e.\
% $ \, \cO_M\big(|M|\big) \, $  for any supermanifold  $ M $.  In addition,
%

 A key feature of Lie supergroups is that they admit a global splitting.
 Namely a Lie supergroup $G$  over $\bk \in \{\br,\bc\}$ of dimension $p|q$ can be expressed as
%
% superalgebra
%
% isomorphisms
%%%%%%%%%
% $ \; \cO_G\big(|G|\big) \, \cong \, \cO_{G_\zero}\!\big(G_\zero\big) \mathop{\otimes}\limits_\bk \Lambda_\bk(\xi_1,\dots,\xi_q) \; $
%%%%%%%%%
%%%%%%%%%
%   for any supergroup  $ G $   --- with  $ q $  such that the super dimension of  $ G $  is  $ \, p|q \, $.  Geometrically, this means that there exist supermanifold splittings
\[  G \, \cong \, G_\zero \times \bigwedge_q^\bk  ,\]
where $G_\zero$ is a Lie group of dimension $p$,
and $\bigwedge_q^\bk$  is the supermanifold given by
the singleton $\{*\}$ together with
%  $\big|\bigwedge_q^\bk\big| := \{*\}$  and
$\cO_{\bigwedge_{\,q}^\bk}\big(\{*\}\big) := \Lambda_\bk(\xi_1,\dots,\xi_q)$,
i.e. a single point endowed with a purely odd,  $ q $-dimensional  affine superstructure.
See for instance  \cite{ccf}  and  \cite{Ga1,Ga2}.

\smallskip

%%%%%%%%%
%%%%%%%

\subsection{Unitarity}  \label{subsec: unitary rep.'s}
%
%%%%%%%%%

 Let  $ V $  be a complex super vector space.
 A super Hermitian metric on
 $ V $  is a map  $ \, B : V \times V \lra \bc \, $  which is  $ \bc $--linear (resp.\  $ \bc $--antilinear)  in the first (resp.\ the second) entry, such that for all non-zero homogeneous vectors  $ u , v \in V $,
\beq
  \begin{array}{cl}
    \mbox{(a)}  &  B(u,v) = 0  \mbox{\ \  if  $ \, |u| \neq |v| \, $}  \quad \quad \mbox{\ \ \textsl{(consistent)}}  \\
    \mbox{(b)}  &  B(u,v) = (-1)^{|u| \cdot |v|} \overline{B(v,u)}  \qquad
\mbox{\ \ \textsl{(super Hermitian symmetric)}}  \\
    \mbox{(c)}  &  B(v,v) \in i^{|v|} \br_+ \qquad \qquad \qquad  \mbox{\ \ \textsl{(super positive)}}.
       \label{abi}
  \end{array}
\eeq
 In  \eqref{abi}(c),  we make the convention that
\beq
  i^{|v|} \in \{ i^0 , i^1 \} = \{1,i\}  \label{ico}
\eeq
 to avoid ambiguity arising from  $ \, i^{[2]} \, $  and so on   ---  see \cite[\S4.1]{fg}.

The set $\fg \fl(V)$ of all linear maps from $V$ to itself
 is naturally a complex Lie superalgebra.
Let
 %the super analogue of the skew-adjoint operators,
\begin{equation}  \label{eqx}
  \hskip-5pt   \mathfrak{u}_B(V)  \, = \,  \{ u \!\in\! \fgl(V)\;;\; B(u(v) , w) + {(-1)}^{|u||v|} B( v , u(w)) = 0 \, , \; v, w \in V_\zero \cup V_\one \}.
\end{equation}
 This is a real subalgebra of $\fg \fl(V)$.
A unitary representation of $\fg$ on $V$ is a Lie superalgebra morphism
$ \fg \lra \fu_B(V) $.

 Observe that  $ B $  (resp.\  $ -iB  $)  is positive definite on  $ V_\zero $
%%%%%%%%%
   (resp.  on  $ V_\one $).
%%%%%%%%%
%%%%%
% \,so  $ \, H \oplus (-iH) \, $
%%%%%
 Therefore,  $ B\big|_{V_\zero} \! \oplus (-iB)\big|_{V_\one}  $
%%%
 is an ordinary inner product on  $ V = V_\zero \oplus V_\one $.

\begin{definition}
 {\rm A super Hilbert space is a super vector space  $ V $
 equipped with a super Hermitian metric  $ B$  such that the ordinary
   inner product  $ \; B\big|_{V_\zero} \! \oplus (-iB)\big|_{V_\one} \; $
   makes  $ \, V = V_\zero \oplus V_\one \, $  into a complete metric space.}   \label{suhi}
\end{definition}

Let $V$ be a super Hilbert space with metric $B$.
Let $U_B(V)$ denote the group of all automorphisms on $V$ which preserve $B$.
Let  $  \fg = {\rm Lie}(G)  $.
A unitary representation of $G_\zero$ is a group homomorphism
 $ \rho : G_\zero \relbar\joinrel\lra U_B(V) $.
 Its differential gives a unitary Lie algebra representation
 on the smooth vectors $V^\infty$, namely
 $ d \rho : \fg_\zero \relbar\joinrel\lra \fu_B(V^\infty)  $
 (see for instance \cite[\S4]{ns}).

   A unitary representation of a Lie supergroup $G$ on a super Hilbert space $V$
 consists of unitary representations of $G_\zero$ and  $ \fg $,
\beq
  \rho : G_\zero \relbar\joinrel\lra U_B(V) \;\; ,  \quad
 \pi : \fg \relbar\joinrel\lra \fu_B(V)  \label{cmpt}
\eeq
 such that  $\rho$ and $\pi$ are compatible
%%%%%%%%%
   in the sense that $d \rho = \pi$ on $V^\infty$
 \cite[Def.4.2.1]{ns}.
 \vskip7pt
   The following are basic results about unitary representations,
%
%%%%%
% (of Lie superalgebras or Lie supergroups),
%%%%%
%
 which are proved exactly as in the non-super framework:

\vskip7pt

\begin{lemma}  \label{lemma: orthog-unitary}
 Let  $ V $  be a unitary representation of a complex Lie superalgebra or a complex Lie supergroup
 with respect to some form  $ B  $.
   If  $ W $  is a subrepresentation of  $ \, V $,  then its orthogonal space
  $$  W^\perp  \, := \,  \{ v \in V \;\;;\;\; B(v,w) = 0 = B(w,v) \;\; \forall \, w \in W \}  $$
 is a subrepresentation as well.
\end{lemma}

\vskip7pt

A representation is said to be completely reducible if it is the direct sum of irreducible subrepresentations.
   Lemma \ref{lemma: orthog-unitary}, through induction, yields the following:

\vskip9pt

\begin{proposition}  \label{prop: unitary => semisimple}
 Every finite dimensional unitary representation of a complex Lie superalgebra or a complex Lie supergroup is completely reducible.
\end{proposition}

 \vskip7pt
%
%%%%%%%%%

%%%%%%%

\subsection{Abelian connected Lie supergroups}  \label{Abel-conn-sgroups}
 We shall work with a connected Abelian real Lie supergroup
 \[ G := T_n \times \R^m \times \bigwedge^\R_k.\]
 Here
 $G_{\bar{0}}  =  T_n \times \R^m$, and
%
 %these are normal subsupergroups, for which hereafter we will adopt standard identifications  $G_{\bar{0}}  \cong T_n \times \R^m  $  and  $ G_{\bar{1}} \cong \bigwedge^\R_k \, $.  On the other hand, in a general Lie supergroup  $ G' $  there exists no canonical analogue of the subsupergroup  $G_{\bar{0}}$  that is a natural counterpart of  $ G'_\zero \, $:  indeed  $ G_{\bar{1}}$  is a specific peculiarity of the case under study.
$ G_{\bar{1}} = \bigwedge^\R_k $  is the spectrum
%%%%%%%%%
(in the sense of  \cite[Ch.10]{ccf})
%%%%%%%%%
%%%%%%%%%
 of  $\Lambda_\R^k( \xi )  := \Lambda_\R(\xi_1\,,\dots,\xi_k)  $, the real Grassmann algebra in  $ k $  generators
 $ \xi_1 ,..., \xi_k $.

   Both $G_{\bar{0}}  $  and  $ G_{\bar{1}} $  are connected Abelian real Lie supergroups on their own, with  $G_{\bar{0}} $  entirely even and  $G_{\bar{1}} $  entirely odd.
%
%    --- as so are (entirely even and entirely odd, respectively) their tangent Lie superalgebras.
% Indeed, for the latter
%
Since  $ G = G_\zero \times G_\one$  is a direct product of
Lie supergroups,
at the tangent level, $\Lie(G) = \fg = \fg_\zero \oplus \fg_\one $ is a direct sum of Lie superalgebras.
   In the sequel, we denote by  $ M $  the complexification of  $ G \, $:  hence  $ M $  is a complex Lie supergroup, whose tangent Lie superalgebra is  $ \Lie(M) = \fg \otimes \bc$.

% \newpage
\bigskip
   \bigskip

 %%%%%%%%%%%%%%%%%%%%%%%%%%%%%%%%%%%%%%%%%%%%%%%%%%%%%%%%%%%%
%%%%%%%%%%%%%%%%%%%%%%%%%%%%%%%%%%%%%%%%%%%%%%%%%%%%%%%%%%%%%
%%  SECTION 3
%%%%%%%%%%%%%%%%%%%%%%%%%%%%%%%%%%%%%%%%%%%%%%%%%%%%%%%%%%%%%
 %%%%%%%%%%%%%%%%%%%%%%%%%%%%%%%%%%%%%%%%%%%%%%%%%%%%%%%%%%%%

\section{Irreducible unitary representations}
\setcounter{equation}{0}
\label{irredu}

   Let  $ G =  T_n \times \br^m \times {\textstyle \bigwedge}^\br_k $  be the real Abelian
    Lie supergroup with super dimension $ (n+m)|k  $,  where
 $ G_\zero = T_n \times \br^m$ is the product of the real torus $T_n$ and $\br^m$, and
 ${\textstyle \bigwedge}^\br_k $ is the spectrum of the Grassmann algebra
 $  \Lambda_\br(\xi_1,\dots,\xi_k) $ --- cf. \S2.4 above.

                                                                                \par
We introduced the unitary representations in (\ref{cmpt}).
% (see \cite[Def.4.2.1]{ns} for details).
   Let  $  \widehat{G} $  be the set of all irreducible unitary representations of  $ G $,  up to equivalence.
 In this section, we prove  Theorem \ref{dual2},
  which yields isomorphisms
  \[ \wh{G} \,\cong\, \wh{G_\zero}  \times \bz_2 \,\cong\, \bz^n \times \br^m \times \bz_2 .\]
    It identifies $  \wh{G}  $  with the set of pairs  $  (\la_1 , \la_2  , \epsilon)  $,  where  $ \la_1 \in \bz^n $,
   $ \la_2 \in \br^m  $,    and  $ \ep \in \bz_2  $  denotes the parity of the representation space.
   We also provide a realization of the representation space in terms of holomorphic functions on $M_\zero$;
   see  Example \ref{xam}.
Let $\fg = \gz + \go := {\rm Lie}(G)$ be the tangent Lie superalgebra of $G$.

\begin{proposition}  \label{shee}
   If  $ \, V \! $  is a unitary representation of  $ \fg $,  then $\fg_\one$ acts trivially on $V$.
\end{proposition}
\begin{proof}
Let $B$ be a super Hermitian metric on $V$, and let
 $  \pi : \fg \relbar\joinrel\lra \fu_B(V)  $  be a Lie superalgebra homomorphism
 (see (\ref{abi}) and (\ref{eqx})).  Let  $ \xi \in \fg_\one  $, and let  $  A = \pi(\xi)  $.
 Then  $ A $  is odd, namely it reverses the parities of  $ V $. Since  $ \fg $  is Abelian,
 we have  $  [\xi,\xi] = 0  $, and so
\beq
  0 = \pi\big([\xi,\xi]\big) \, = \, \big[ \pi(\xi) \, , \pi(\xi) \big] \, = \, [A,A] \, = \, 2 A^2 .
\label{aitt}
\eeq

Let $v \in V_\zero \cup V_\one$ be a homogeneous vector.
We have
\[
\begin{array}{rll}
B(Av,Av) &  = \;\;  (-1)^{|v|} B(A^2v,v)  &  \quad  \mbox{by (\ref{eqx})} \\
&  = \;\;  0 . &  \quad  \mbox{by (\ref{aitt})}
\end{array}
\]
Since $B$ is super positive, this implies that $Av = 0$ for all
$v \in V_\zero \cup V_\one  $,  hence  $ A = 0  $.
Therefore  $ \pi(\xi) = 0 $  for all  $ \xi \in \fg_\one  $,  i.e.  $ \fg_\one $
acts trivially on  $ V $.
\end{proof}

\vskip13pt

\noindent {\it Proof of Theorem \ref{dual2}:}

%%%%%%%%%
%%%%%%%
Let $V$ be a super Hilbert space with metric $B$,
and let $\fu_B(V)$ be the real subalgebra (\ref{eqx})
of $\fg \fl(V)$.
Let $\rho$ be a unitary representation of $G$ on $V$.
It leads to a morphism of Lie superalgebras,
$d\rho : \fg \lra \fu_B(V^\infty)$, where $V^\infty$ are the smooth vectors
(see \S 2.3).
By Proposition \ref{shee},
$ V^\infty $  is a $ G_\zero $-module  equipped with trivial  $ \fg_\one $-action.

Suppose in addition that the $G$-representation on $V$ is irreducible.
We claim that $V^\infty$ is an irreducible $G_\zero$-module.
Assume otherwise, namely $V^\infty$ has a nontrivial
 $ G_\zero $-submodule $W$. Since $\fg_\one$ acts trivially on $V^\infty$,
 this means that $W$ is
 also a nontrivial $G$-subrepresentation of $V$,
 which contradicts the irreducibility of $V$.
Hence $V^\infty$ is an irreducible $ G_\zero $-module as claimed.

 %We know from general representation theory that
 All irreducible unitary representations
 of connected Abelian Lie groups are 1-dimensional.
 So $V^\infty$ is 1-dimensional.
 Since $V^\infty$ is dense in $V$, it implies that $V = V^\infty$.
So every irreducible unitary  $ G $-representation is 1-dimensional.
 Hence
\beq
  \wh{G} = \wh{G_\zero} \times \bz_2 .
  \label{spq}
\eeq

In  \eqref{spq},  the  $ \bz_2 $  component controls the parity.  We have $G_\zero = T_n \times \br^m$, so
the  $ \wh{G_\zero} $  component amounts to the integral weights  $ \la \in \bz^n \times \br^m $,
or equivalently their characters  $\chi_\la  $,  see  \eqref{niw}.
We write  $  V_\la^\ep \in \wh{G}  $  accordingly, where  $ \, \ep \in \bz_2 $.
Its elements  $ v \in V_\la^\ep$  satisfy  $ \, t \cdot v = \chi_\la(t) v $  for all  $  t \in G_\zero  $.

For integral weights  $ \la $  and  $ \mu  $,  their characters satisfy  $ \chi_\la \chi_\mu = \chi_{\la + \mu}  $.
So the tensor product of representations leads to  $ V_\la^\ep \otimes V_\mu^\de = V_{\la + \mu}^{\ep \de}  $, and
 yields a group isomorphism  $  \wh{G} \cong \bz^n \times \br^m \times \bz_2 $.
This proves  Theorem \ref{dual2}.   \hfill  $ \Box $

\vskip9pt

   The following example provides a realization of the representation space associated with  $  (\la\,, \ep) \in \wh{G}  $.

\vskip9pt

\begin{example}  \label{xam}
 {\rm Recall that $M_\zero = \bc^n/i \bz^n \times \bc^m$. Fix $\la \in \bz^n \times \br^m$,
 identified with an integral weight in $i(\bz^n + \br^m)^*$, see (\ref{niw}).
Consider the holomorphic function
 \[ f : M_\zero \lra \bc \;,\; f(z) = e^{\la (z)} = e^{\la(x+iy)} ,\]
 where $z = x+iy$ are the coordinates in (\ref{loc-coord}).
 Let $V_\la = \{c f \;;\; c \in \bc\}$.

 The action of $G_\zero = \br^n/i\bz^n \times \br^m$ on $M_\zero$
 lifts to a $G_\zero$-action on $V_\la$ by
 \[ (r \cdot f)(z) = e^{\la (z+r)} = \chi_\la(r) f(z) \;;\; r \in G_\zero \]
 We can provide a Hermitian metric on $V_\la$ by
 \beq \langle cf , cf \rangle = |c|^2 \label{fmo}
 \eeq
 for all $c \in \bc$. Since $|\chi_\la(r)| = 1$, we have
 \[ \langle r \cdot f , r \cdot f \rangle = \langle \chi_\la(r) f , \chi_\la(r) f \rangle
 = \langle  f ,  f \rangle \]
 so the $G_\zero$-action on $V_\la$ is unitary.
 Depending on whether we declare it to be even or odd,
  $V_\la$ is the member of $\wh{G}$
 parametrized by $(\la, \pm) \in (\bz^n \times \br^m) \times \bz_2$
   \hbox{in Theorem \ref{dual2}.}

Let $\ch(M_\zero)$ be the space of holomorphic functions on $M_\zero$.
As is often the case in harmonic analysis,
instead of the formal Hermitian metric (\ref{fmo}),
one may
construct an $L^2$-structure on $f \in \ch(M_\zero)$ by
$\int_{M_\zero} f \bar{f} \, \mu$, where $\mu$
is a $G_\zero$-invariant measure.
The functions which converge in this $L^2$-structure form
a unitary $G_\zero$-representation on $V_\la$.
We will consider this issue later, from the viewpoint of
geometric quantization.
}
\label{contoh}
\end{example}

% \newpage
\bigskip
   \bigskip

 %%%%%%%%%%%%%%%%%%%%%%%%%%%%%%%%%%%%%%%%%%%%%%%%%%%%%%%%%%%%
%%%%%%%%%%%%%%%%%%%%%%%%%%%%%%%%%%%%%%%%%%%%%%%%%%%%%%%%%%%%%
%%  SECTION 4
%%%%%%%%%%%%%%%%%%%%%%%%%%%%%%%%%%%%%%%%%%%%%%%%%%%%%%%%%%%%%
 %%%%%%%%%%%%%%%%%%%%%%%%%%%%%%%%%%%%%%%%%%%%%%%%%%%%%%%%%%%%

\section{Super K\"ahler structures}
\setcounter{equation}{0}
\label{superkahler}

   In this section, we set up the notions of super symplectic forms and
   super K\"ahler forms (see for instance \cite{vars}),
   together with Hamiltonian actions and moment maps.
   We then apply these notions to our setting and prove Theorem \ref{kah2}.

   Let $M$ be a complex supermanifold.
Let $\Omega^p(M)$ denote the set of all differential $p$-forms of $M$
with complex coefficients.
Let $\Om^{1,0} (M,\bc)$ (resp.  $\Om^{0,1} (M,\bc)$) be
    the complex 1-forms that are eigenvectors of its complex structure with eigenvalue $i$ (resp. $-i$).
    Their exterior product leads to the $(1,1)$-forms $\Om^{1,1} (M,\bc) \subset \Omega^2(M)$.

\begin{definition} {\it Super K\"ahler forms and super symplectic forms}

\noindent {\rm (a)} \, Let $M$ be a complex supermanifold.
A  \textit{super K{\"a}hler form}  on  $ M $  is a closed real $(1,1)$-form
$\omega \in \Omega^{1,1}(M)$ which is super skew-symmetric and positive definite.

\noindent {\rm (b)} \, Let $M$ be a real supermanifold.
   A  \textit{super symplectic form}  on  $ M $  is a closed 2-form
$\omega \in \Omega^2(M)$ which is super skew-symmetric and nondegenerate.

    We also assume that $\om$ is \textit{consistent}, namely  $ \, \omega(X,Y) = 0 \, $  whenever  $ \, |X| \not= |Y| \, $.
\label{dska}
\end{definition}

Here closed means $d \om = 0$, real means $ \overline{\omega} = \omega $,
and super skew-symmetric means $ \omega(X,Y) = -{(-1)}^{|X|\,|Y|} \, \omega(Y,X) $.

   \vskip5pt

In classical geometry,
Hamiltonian actions and moment maps play important roles in
symplectic manifolds $M_\zero$ with actions of Lie groups $G_\zero$ \cite[\S26]{gs}.
In particular, for $z \in \gz$, we define the infinitesimal vector field $z^\sh$ on $M_\zero$ by
\beq
  (z^\sh f)(m)  \, := \,  \bigg( \frac{d}{dt} f\big(e^{tz}.m\big) \!\bigg)\bigg|_{t=0}
  %\;\; ,  \quad f \in C^\infty\big(M \big) \; , \;\; m \in M \; ,
  \label{abyz}
\eeq
for all $f \in C^\infty(M_\zero)$ and $m \in M_\zero$.
We now extend these notions to the super setting.

\smallskip

Let $G$ be a Lie supergroup acting on a symplectic supermanifold  $(M,\om)$.
Given $z \in \fg$, we define
the infinitesimal vector field $z^\sh$ on $M$
as in \cite[Def.8.2.1]{ccf},
namely $z^\sh = \rho_a(z)$ by its notation.
The action of $z^\sh$ can locally be described by (\ref{abyz}) for $f \in C^\infty(M)$ and $m \in M$.
This is
interpreted in the functorial language, i.e. in terms of the functor of points of $G$ and $M$ alike.
The super exponential $(t \mapsto e^{tz})$ has a simpler form for our
Abelian Lie supergroup; see \cite[\S 7.5]{ccf}.
With these assumptions, we let $\imath(z^\sh) \om$ denote the 1-form $\om(z^\sh, \cdot)$.

\begin{definition} {\it Hamiltonian actions and moment maps}

 Let $G$ be a Lie supergroup acting on a symplectic supermanifold  $ (M,\om)  $.
 We say that the action is {\it Hamiltonian} if there exists a map
 $ \Phi : M \lra \fg^* $,  known as the {\it moment map},
 such that

 \noindent {\rm (a)} $\,$ $ d(\Phi, z) = -\imath(z^\sh)\om $  for all  $z \in \fg$,

 \noindent {\rm (b)} $\,$ $\Phi$ is $G$-equivariant.   \label{momm}
\end{definition}

In Definition \ref{momm}(a),
 we treat  $ (\Phi,z) $  as a member of  $ C^\infty(M) $  by  $  m \mapsto (\Phi(m))(z) \in \br $,
 so that $  d(\Phi,z) \in \Om^1(M)  $.
The moment map may not exist, because one needs the closed 1--form  $  \imath(z^\sh)\om  $  to be exact.
So an obstruction is the cohomology  $ H^1(M) $, for example the ordinary 2-torus
acting on itself preserving the invariant volume form has no moment map.

In Definition \ref{momm}(b),
$G$-equivariance means that $\Phi$ intertwines the $G$-action on $M$ with its
coadjoint action on $\fg^*$.
In our case $G$ is Abelian, so it acts trivially on $\fg^*$,
and hence $\Phi$ is $G$-invariant.

\medskip

We now study super K\"ahler geometry and Hamiltonian actions in our setting, so
let  $ \, G = T_n \times \br^m \times \bigwedge^\br_k \, $  be an Abelian Lie supergroup,
and let  $ M $
be its complexification  (\ref{split2}).
   We use the holomorphic coordinates  $  z = x+i\,y  $  and $\zeta = \xi + i \, \eta$ on $M$
   as given in  \eqref{loc-coord},
   where  $ G $  acts on the imaginary part.

    We briefly recall some basic facts on chain complexes.
  A chain complex $C$ is a sequence of linear maps $d : C^n \lra C^{n+1}$
for all $n \geq 0$,
such that $d^2 =0$.
Let $Z^n \subset C^n $ denote the kernel of $d$.
Suppose that $C_a$ and $C_b$ are two chain complexes.
 Then $C = C_a \otimes C_b$
is again a chain complex under  $ \, d \otimes 1 + 1 \otimes d \, $,
where  $ \, C^n = \sum_{i+j=n} C_a^i \otimes C_b^j \, $.
In this case,
$Z^n = {\textstyle \sum_{i+j=n}} Z_a^i \otimes Z_b^j$.
 In particular, for  $ \, n=2 \, $,
\beq
Z^2 = Z_a^0 \otimes Z_b^2 + Z_a^1 \otimes Z_b^1 + Z_a^2 \otimes Z_b^0 .
 \label{pen}
\eeq

We apply the chain complex theory to our setting. Express
the differential forms of our
complex Abelian Lie supergroup $M$ as tensor product of chain complexes,
  \beq
  \begin{array}{l}
   C \;=\; \Om^\bullet (M,\bc)  \; = \;  \Om^\bullet(M_\zero, \bc) \otimes {\textstyle \bigwedge}_{\xi,\,\eta}^\bullet  \;, \\
   C_a^n \;=\; \Om^n (M_\zero,\bc) = \big\{\, {\textstyle \sum_{|I|+|J| =n}} f_{IJ}(x,y) \, dx^I \wedge dy^J \,\big\} ,\\
   C_b^n \;=\; {\textstyle \bigwedge}_{\xi,\,\eta}^n = \big\{\, {\textstyle \sum_{|I|+|J|=n}} f_{IJ}(\xi,\eta) \, d\xi^I \wedge d\eta^J \,\big\},
    \end{array}
    \label{pch}
    \eeq
where $f_{IJ}(x,y)$ are smooth functions on $M_\zero$,
and $f_{IJ}(\xi,\eta)$ are polynomials in which every monomial contains each $\xi_i$ and $\eta_j$ at most once.
Hereafter we will refer to such $f_{IJ}(\xi,\eta)$ as {\it Grassmann polynomial}.
Here $ \, \bigwedge_{\xi,\,\eta}^\bullet \, $  is also a chain complex with exterior derivative  $ d $    \cite[p.234]{cw}.

Let $\om$ be a consistent,  $ G $-invariant  super K\"ahler form on $M$.
So in (\ref{pch}), $\om \in Z^2$.
Since $\om$ is consistent, we can ignore
$Z_a^1 \otimes Z_b^1$ of (\ref{pen}). Also, $Z_a^0 = Z_b^0 = \bc$. So by
(\ref{pen}), we write
\beq
 \om = \om_\zero + \om' \in Z_a^2  +  Z_b^2  .\label{jad}
 \eeq

 We study $\om_\zero$ and $\om'$ of (\ref{jad}) separately.
 We start with $\om_\zero$,
 which is $G_\zero$-invariant K\"ahler form on $M_\zero$.
 Since  $ G_\zero $  has dimension  $ n+m $,
   the indices $p,q,j,k$ below are summed over $1,...,n+m$.
   Here $F = F(x)$, and $F''$ is its Hessian matrix.

\vskip9pt

\begin{proposition}  \label{kf}
A closed $G_\zero$-invariant real $(1,1)$-form on $M_\zero$ can be expressed as
\beq  \om_\zero  \,\; = \;\,  \sum_{p,q} \, \frac{\pa^2 F}{\pa x_p \, \pa x_q} \, dx_p \we dy_q
  \, + \, \sum_{j<k} c_{jk} \big( dx_j \we dx_k + dy_j \we dy_k \big)  ,
  \label{expr}
  \eeq
  where $C = (c_{jk})$ is a skew-symmetric real matrix.
It is K\"ahler if and only if the symmetric matrix
 $\left( \begin{array}{cc}
 F'' & -C \\
 C & F''
 \end{array}
 \right)$
 is positive definite everywhere.
%It has a $G_\zero$-invariant potential function if and only if $C=0$.
\end{proposition}
\begin{proof}
Let $\om_\zero$ be a closed $G_\zero$-invariant real $(1,1)$-form on $M_\zero$.
We write
\beq
 \om_\zero = \sum_{j,k} \phi_{jk} dz_j \we d \bar{z}_k
 = \sum_{j,k} \phi_{jk} (dx_j +i dy_j) \we (dx_k -i dy_k) ,
\label{mux}
 \eeq
where  $ \, \phi_{jk}=\phi_{jk}(x) \, $  by $G_\zero$-invariance.
Let $h_{pq} = -i(\phi_{pq}+ \phi_{qp})$ and $c_{jk} = \phi_{jk}-\phi_{kj}$,
so $h_{pq}= h_{qp}$ and $c_{jk}= -c_{kj}$.
The wedge product is skew-symmetric, so (\ref{mux}) gives
\[
 \om_\zero  = \sum_{p,q} h_{pq} dx_p \we dy_q + \sum_{j<k} c_{jk} (dx_j \we dx_k + dy_j \we dy_k)
  = \om_A + \om_B ,
 \]
where $\om_A$ (resp. $\om_B$) denotes the summands containing $h_{pq}$ (resp. $c_{jk}$).

Since $\om_\zero$ is closed, we have $0 = d \om_\zero = d\om_A + d\om_B$.
One checks that $d\om_A$ is the sum of terms with $dx_r \we dx_p \we dy_q$
whereas $d\om_B$ is the sum of other types of 3-forms,
so there is no nontrivial cancellation between them and we get $d\om_A = d\om_B =0$.

For $\om_A$,
we have $h_{pq}= h_{pq}(x)$ and
\[ 0 = d\om_A = \sum_q d \big( \sum_p h_{pq} \, dx_p \big) \we dy_q .\]
So for each $q$, $\sum_p h_{pq} dx_p$ is a closed 1-form of $\br^{n+m}$.
Hence it is exact because
$\br^{n+m}$ has trivial deRham cohomology.
Namely,  there exists $f_q = f_q(x)$ such that
\[ d f_q = \sum_p h_{pq} \, dx_p . \]
We have $\frac{\pa f_q}{\pa x_p} = h_{pq} = h_{qp} = \frac{\pa f_p}{\pa x_q}$,
so there exists  $  F \in C^\infty(\br^{n+m})  $ such that
 $  \frac{\pa F}{\pa x_q} = f_q  $. It implies that
 $  \frac{\pa^2 F}{\pa x_p \pa x_q} = h_{pq}  $,  so
\[ \om_A = \sum_{p,q} \frac{\pa^2 F}{\pa x_p \pa x_q} dx_p \we dy_q . \]
We have proved that $\om_\zero$ has the expression (\ref{expr}).

For $\om_\zero$ to become K\"ahler, it remains only to check
positivity, namely $\om_\zero (v,Jv) > 0$
for any non-zero tangent vector $v$,
where $J$ is the almost complex structure.
Write
$v = \sum_r (a_r \frac{\partial}{\partial x_r} + b_r \frac{\partial}{\partial y_r})$.
We have $J \frac{\partial}{\partial x_r} = \frac{\partial}{\partial y_r}$ and
$J \frac{\partial}{\partial y_r} = -\frac{\partial}{\partial x_r}$, so
\[
\begin{array}{rcl}
 \om_\zero(v,Jv)
  & = & (\om_A+\om_B) \big(\sum_r (a_r \frac{\partial}{\partial x_r} + b_r \frac{\partial}{\partial y_r}),
\sum_s (a_s \frac{\partial}{\partial y_s} - b_s \frac{\partial}{\partial x_s}) \big) \\
 & = & \om_A \big(\sum_r a_r \frac{\partial}{\partial x_r}, \sum_s a_s \frac{\partial}{\partial y_s} \big)
+ \om_A \big( \sum_r b_r \frac{\partial}{\partial y_r},- \sum_s b_s \frac{\partial}{\partial x_s} \big) \\
& & + \om_B \big(\sum_r a_r \frac{\partial}{\partial x_r}, -\sum_s b_s \frac{\partial}{\partial x_s} \big)
+ \om_B \big(\sum_r b_r \frac{\partial}{\partial y_r}, \sum_s a_s \frac{\partial}{\partial y_s} \big)  \\
& = & F'' (a,a) + F''(b,b) +2C (a,b) \\
& = & \left( \begin{array}{cc}
 F'' & -C \\
 C & F''
 \end{array}
 \right)
  \left( \begin{array}{c} a \\
 b
 \end{array} \right)
 \left( \begin{array}{c} a \\
 b
 \end{array} \right) .
\end{array}
\]

Here $a =(a_r),b=(b_s) \in \br^{n+m}$.
In the last expression,
$\left( \begin{array}{cc}
 F'' & -C \\
 C & F''
 \end{array}
 \right)$
 is a symmetric matrix of size $2(n+m) \times 2(n+m)$
 regarded as a bilinear form on $\br^{2(n+m)}$ (see (\ref{fhua})),
 so we evaluate it on $((a,b),(a,b)) \in \br^{2(n+m)} \times \br^{2(n+m)}$.
 It leads to the positivity of $\om_\zero$
 in the second statement of this proposition.
\end{proof}

%Finally, we consider the existence of a $G_\zero$-invariant potential function for $\om_\zero = \om_A+\om_B$.
%Suppose that $\om_\zero$ has a $G_\zero$-invariant potential function, say $\om_\zero = i \pa \bar{\pa} H$ where $H =H(x)$. Then
%\beq \begin{array}{rl}
%\om_\zero & = i \sum_{j,k} \frac{\pa}{\pa z_j} \frac{\pa}{\pa \bar{z}_k} H dz_j \we d\bar{z}_k \\
%& = \frac{i}{4}\sum_{j,k} \frac{\pa^2 H}{\pa x_j \pa x_k} (dx_j + i dy_j) \we (dx_k -i dy_k) \\
%& = \frac{i}{4}\sum_{j,k} \frac{\pa^2 H}{\pa x_j \pa x_k} (-2i \, dx_j \we dy_k + dx_j \we dx_k + dy_j \we dy_k).
%\end{array}\label{ahii} \eeq
%The skew symmetric properties means that $dx_j \we dx_k$ and $dx_k \we dx_j$ cancel off in pairs,
%and likewise for $dy_j \we dy_k$. So the last expression of (\ref{ahii}) becomes
%\beq \om_\zero = i \pa \bar{\pa} H = \frac{1}{2}\sum_{j,k} \frac{\pa^2 H}{\pa x_j \pa x_k} dx_j \we dy_k .\label{chau} \eeq
%Hence if $\om_\zero$ has a $G_\zero$-invariant potential function, then $\om_B=0$.
%Conversely, if $\om_B=0$, then by (\ref{expr}) and (\ref{chau}), we get $\om_\zero = i \pa \bar{\pa}H$ where $H=2F$.

The above proposition fixes an error in \cite[Thm.1.1]{jga},
which overlooked the component $\sum_{j<k} c_{jk} ( dx_j \we dx_k + dy_j \we dy_k )$.

\medskip

Next we consider $\om' \in Z_b^2$ of (\ref{jad}).
 We state it in the following proposition.

\begin{proposition}  \label{skf}
A closed $G$-invariant real $(1,1)$-form on $M$
has the expression  $\om_\zero + \om'$ with $\om_\zero$ is given by (\ref{expr}), and
\beq
 \om'  \; = \;  \sum\nolimits_{p,q} \, k_{pq} \, \big(\, d\xi_p \, d\xi_q + d\eta_p \, d\eta_q \,\big)
   \, + \, \sum\nolimits_q \, dh_q \, d \eta_q  , \label{hhyy}
   \eeq
   where
  $K = (k_{pq})$ is a symmetric real matrix, $h_q = h_q(\xi)$
  are Grassmann polynomials in $\xi$, and $H= \big( \frac{\pa h_p}{\pa \xi_q} \big)$
  is a skew-symmetric matrix. It is K\"ahler if and only if $\om_\zero$ is a K\"ahler form on $M_\zero$,
  and the symmetric matrix
  $\left( \begin{array}{cc} 2K & -H \\
 H & 2K
 \end{array} \right)$
 is positive definite everywhere.
\end{proposition}
\begin{proof}
By (\ref{jad}), we have $\om_\zero + \om' \in Z_a^2 + Z_b^2$, and
Proposition \ref{kf} has studied $\om_\zero$. So we now consider $\om'$.
Since $M$ has dimension $(n+m)|k$, all summations in this proposition are made over $1,...,k$.
Since $\om'$ is a real $(1,1)$-form,
 we use the coordinates $ \zeta_p = \xi_p + i \eta_p $ of (\ref{loc-coord}) and write
\[
\om' \, = \, {\textstyle \sum_{p,q}} \, f_{pq}(\xi) \, d \zeta_p \, d\bar{\zeta}_q \in Z_b^2 .
\]
Here  $ f_{pq}(\xi) $ is  independent of  $ \eta $  because of $ G $-invariance.  Then
\beq \label{tff}
   \overline{\om'}  \;  =  {\textstyle \sum_{p,q}} \, \bar{f}_{pq} d \bar{\zeta}_p \, d \zeta_q  \, =
   {\textstyle \sum_{p,q}} \, \bar{f}_{pq} \, d \zeta_q \, d \bar{\zeta}_p  \, =  {\textstyle \sum_{p,q}} \, \bar{f}_{qp} \, d \zeta_p \, d \bar{\zeta}_q
\eeq
 where we take into account  $  d \zeta_q \, d \bar{\zeta}_p  =  d \bar{\zeta}_p \, d \zeta_q $
 (as  $ \omega $  is super skew-symmetric,  $ \omega' $  in turn is symmetric).  By  $  \overline{\omega'} = \omega'  $  and  \eqref{tff},  we have
 $  \bar{f}_{qp} = f_{pq}  $ for all  $ p  , q  $.

Let us write  $f_{pq} = k_{pq} + i \, h_{pq}$ as real and imaginary parts;
\beq  \label{spa}
k_{pq} = \frac{1}{2} \big( f_{pq} + \bar{f}_{pq} \big) =
\frac{1}{2} \big( f_{pq} + f_{qp} \big)
\;\;,\;\;
h_{pq}  = - \frac{i}{2} \,\big( f_{pq} - \bar{f}_{pq} \big) =
  \frac{i}{2} \big(- f_{pq} + f_{qp} \big) .
\eeq

We now expand  $  \zeta_s = \xi_s + i\,\eta_s  $, keeping in mind that $ d\eta_p \, d\xi_q =  d\xi_q \, d\eta_p$.
By (\ref{spa}), as well as $k_{pq} = k_{qp}$ and $h_{pq} = - h_{qp}$,
\[
  \begin{array}{rl}
     \om'  &   = \,  \sum_{p,q} f_{pq} \, (d \xi_p + i \, d \eta_p) \, (d \xi_q - i \, d \eta_q)    \\
      &   = \,  \sum_{p,q} ( f_{pq} \, ( d \xi_p \, d \xi_q + d \eta_p \, d \eta_q ) \, + \, i \, f_{pq} \,
           (- d \xi_p \, d \eta_q + d \eta_p \, d \xi_q  ) )   \\
            &   = \, \frac{1}{2} \sum_{p,q} ( f_{pq} + f_{qp} ) \,  ( d \xi_p \, d \xi_q + d \eta_p \, d \eta_q )
            - \frac{i}{2} \, \sum_{p,q} ( f_{pq} - f_{qp} ) \, ( d \xi_p \, d \eta_q - d \xi_q \, d \eta_p ) \\
& =  {\textstyle \sum_{p,q}}\, k_{pq} \, ( d \xi_p \, d \xi_q + d \eta_p \, d \eta_q ) \, +
\, {\textstyle \sum_{p,q}}\, h_{pq} \, ( d \xi_p \, d \eta_q - d \xi_q \, d \eta_p ) \\
   &  =  {\textstyle \sum_{p,q}}\, k_{pq} \, ( d \xi_p \, d \xi_q + d \eta_p \, d \eta_q ) \, +
   \,2 \, {\textstyle \sum_{p,q}}\, h_{pq} \, d \xi_p \, d \eta_q \\
   & = \om_X + \om_Y ,
   \end{array}
   \]
where $\om_X$ and $\om_Y$ are the terms containing $k_{pq}$ and $h_{pq}$ respectively.

We add the condition $\om' \in Z_b^2$
and get $0 = d \om' = d\om_X + d \om_Y$.
Here $d\om_X$ (resp.\ $d\om_Y$) contains the terms with  $  d \xi_r \, d\xi_p \, d\xi_q + d \xi_r \, d\eta_p \, d\eta_q  $
(resp.  $  d\xi_r \, d\xi_p \, d\eta_q - d\xi_r \, d\xi_q \, d\eta_p $),
so there is no nontrivial cancellation between them.
Hence $  d\om_X = d\om_Y = 0  $.

For  $ \om_X \, = \, {\textstyle \sum_{p,q}} k_{pq} \, ( d \xi_p \, d \xi_q + d \eta_p \, d \eta_q ) $,
the fact that  $  d\om_X = 0  $  implies that  $ \, k_{pq} \in \br \, $  are all constants, i.e.
   $\om_X =   d (\, {\textstyle \sum_{p,q}} k_{pq} \, ( \xi_p \, d \xi_q + \eta_p \, d \eta_q ) )$ is exact.

  Now we look at  $ \; \om_Y  = 2 \, {\textstyle \sum_{p,q}}\, h_{pq} \, d \xi_p \, d \eta_q  \; $.
 The condition  $ \, d\omega_Y = 0 \, $  implies that
 $\frac{\pa}{\pa \xi_r} (h_{pq}) + \frac{\pa}{\pa \xi_p} (h_{rq}) =0$, so by setting $p=r$, we get
 \beq
  \frac{\pa}{\pa \xi_p} (h_{pq}) =0 , \label{ahdo}
  \eeq
 namely $h_{pq}$ is independent of $\xi_p$.
 Let $h_q = \sum_p h_{pq} \, \xi_p$. By (\ref{ahdo}), $\frac{\pa h_q}{\pa \xi_p} = h_{pq}$, so
 \[ \sum_q dh_q \, d\eta_q = \sum_q \big( \sum_p \frac{\pa h_q}{\pa \xi_p} d\xi_p \big) \, d \eta_q
 = \sum_{p,q} h_{pq} \, d \xi_p \, d \eta_q = \om_Y .\]
 We have shown that $\om'$ has the expression (\ref{hhyy}).

 It remains to study the positivity of $\om'$. We
 consider the $k \times k$ matrices
 \[ K = (k_{pq}) \;\;,\;\; H = (h_{qp}) = \Big( \frac{\pa h_p}{\pa \xi_q} \Big) .\]
 Note that by (\ref{spa}), $K$ is symmetric and $H$ is skew-symmetric.
 We identify them with bilinear forms on $\br^k$ via (\ref{fhua}).
 Let $v = \sum_r (a_r \frac{\pa}{\pa \xi_r} + b_r \frac{\pa}{\pa \eta_r})$, so $a= (a_r), b=(b_s) \in \br^k$.
 We have
 \[
 \begin{array}{rcl}
 \om' (v,v)
 &  = & (\om_X + \om_Y)
 \big( \sum_r (a_r \frac{\pa}{\pa \xi_r} + b_r \frac{\pa}{\pa \eta_r}) ,
 \sum_s (a_s \frac{\pa}{\pa \xi_s} + b_s \frac{\pa}{\pa \eta_s}) \big) \\
 & = & \om_X \big( \sum_r a_r \frac{\pa}{\pa \xi_r} , \sum_s a_s \frac{\pa}{\pa \xi_s} \big)
 + \om_X \big( \sum_r b_r \frac{\pa}{\pa \eta_r} , \sum_s b_s \frac{\pa}{\pa \eta_s} \big) \\
 & & + \, 2 \, \om_Y \big( \sum_r a_r \frac{\pa}{\pa \xi_r} , \sum_s b_s \frac{\pa}{\pa \eta_s} \big) \\
 & = &  2K(a,a) + 2K(b,b) + 2H(a,b) \\
 & = &   \left( \begin{array}{cc} 2K & -H \\
 H & 2K
 \end{array} \right)
 \left( \begin{array}{c} a \\
 b
 \end{array} \right)
 \left( \begin{array}{c} a \\
 b
 \end{array} \right) .
   \end{array}
   \]

In the last expression, $\left( \begin{array}{cc} 2K & -H \\ H & 2K
  \end{array} \right)$
  is a $2k \times 2k$ symmetric matrix, regarded as a bilinear form on $\br^{2k}$.
  So it is positive definite if and only if $\om' (v,v) > 0$ for all $v \neq 0$.
  This proves the proposition.
    \end{proof}

\medskip

We now consider the theory of Hamiltonian group action, discussed in Definition \ref{momm}.
%in our setting, where $G$ is an Abelian Lie supergroup
%acting on its complexification $M$.
We write $\om = \om_\zero + \om'$ according to
Propositions \ref{kf} and \ref{skf}.

\begin{proposition}
Let $\om = \om_\zero + \om'$ be a $G$-invariant super K\"ahler form on $M$.
The $G$-action is Hamiltonian if and only if $C=K=0$.
In this case its moment map is
$\Phi(x,y,\xi,\eta) =
\big( \, \frac{\pa F}{\pa x_p} \,,\, -h_q(\xi) \, \big)$.
\label{phim}
\end{proposition}
\begin{proof}
We use the coordinates  $ (x,y,\xi,\eta) $  in  \eqref{loc-coord}.
   Let  $  u+v \in \gz + \go  $.  Its associated infinitesimal vector field on  $ M $  is
\beq
  u^\sh + v^\sh  \; = \;  \sum_q u_q \, \frac{\pa}{\pa y_q} \, + \, \sum_s v_s \, \frac{\pa}{\pa \eta_s}   \;. \label{tesi}
\eeq
By (\ref{expr}) and (\ref{tesi}), we have
\beq
 \imath(u^\sh) \, \om_\zero = - \sum_{p,q} u_q \, \frac{\pa^2 F}{\pa x_p \pa x_q} \, dx_p \, +
\sum_{j<k} c_{jk} (u_j  \, dy_k - u_k \, dy_j) .\label{angg}
\eeq
 By (\ref{hhyy}) and (\ref{tesi}),  we have
\beq
  \imath(v^\sh) \, \om'  =
 2 \sum_{p,q} k_{pq} v_p \, d \eta_q + \sum_q v_q \, dh_q ,  \label{angk}
\eeq

We want to define the moment map $\Phi : M \lra \fg^*$.
By (\ref{angg}) and (\ref{angk}),
up to addition by a vector, a necessary condition to satisfy Definition \ref{momm}(a) is
\beq
 \Phi(x,y,\xi,\eta) =
\Big( \frac{\pa F}{\pa x_p} - 2 \sum_j c_{jp} y_p \,,\,
-h_q(\xi) - 2 \sum_j k_{jq} \eta_q \Big) \in \fg^* . \label{nowe}
\eeq
Here ``up to addition by a vector'' means that if $\Phi$
satisfies Definition \ref{momm}(a), then so does
$\Phi' (r) := \Phi(r)+ \la$
for all $\la \in \fg^*$, as the constant term vanishes when we apply $d$ to it.
In (\ref{nowe}),
we identify $ \fg^* \cong \br^{(n+m)|k}$, so the even variables have
indices $p=1,...,n+m$, and
the odd variables have indices $q = 1,...,k$.

The variables of $G$ are $(y,\eta)$, and in particular
$y_1,...,y_n$ are the local (but not global) coordinates of $T_n$.
Suppose that $c_{jp} \neq 0$ for some $p=1,...,n$.
Then (\ref{nowe}) is not well-defined, due to the presence of $y_p$.
It implies that $\om$ does not satisfy Definition \ref{momm}(a).

Suppose now that $c_{jp}=0$ for all $p=1,...,n$
(which includes the case where $T_n$ is the trivial group),
so that (\ref{nowe}) is well-defined.
If $c_{jp} \neq 0$ for any $p=n+1,...,n+m$,
or some $k_{jq} \neq 0$,
then (\ref{nowe}) is not $G$-invariant because it contains the $G$-variable
$y_p$ or $\eta_q$.
Consequently, $\om$ does not satisfy Definition \ref{momm}(b).
We have shown that if $C \neq 0$ or $K \neq 0$,
then the action is not Hamiltonian.

Conversely, if $C = K = 0$, then $\Phi$ of (\ref{nowe}) is well-defined and
does not contain the $G$-variables $y$ and $\eta$.
So $\Phi$ satisfies Definition \ref{momm}, and the action is Hamiltonian.
\end{proof}

\sk

\noindent {\it Proof of Theorem \ref{kah2}:}

Theorem \ref{kah2}(a,b) follow from  Propositions \ref{kf} and \ref{skf}.
Theorem \ref{kah2}(c) follows from Proposition \ref{phim}.\hfill  $ \Box $

% \newpage
\bigskip
   \bigskip

%%%%%%%%%%%%%%%%%%%%%%%%%%%%%%%%%%%%%%%%%%%%%%%%%%%%%%%%%%%%%%%%%%
%%%%%%%%%%%%%%%%%%%%%%%%%%%%%%%%%%%%%%%%%%%%%%%%%%%%%%%
%%%%%%%%%%%%%%%%%%%%%%%%%%%%%%%%

\section{Geometric quantization}
\setcounter{equation}{0}
\label{geomquan}

In this section, we prove  Theorem \ref{occ2}.
Fix a $G$-invariant super K\"ahler form $\om$ on  $ M $, where the $G$-action is Hamiltonian with moment map $\Phi$.
By  Theorem \ref{kah2},
\beq \om = d \sum_q \Big( \frac{\pa F}{\pa x_q} \, dy_q + h_q \, d \eta_q \Big) , \label{bbzy}
\eeq
so it is exact.
The positive definite matrices in Theorem \ref{kah2} become
$\left( \begin{array}{cc}
F'' & 0 \\
0 & F''
\end{array} \right)$
and
$\left( \begin{array}{cc}
0 & -H \\
H & 0
\end{array} \right)$.
So in particular $F''$ is positive definite,
namely $F$ is strictly convex.
   We first recall geometric quantization in the ordinary setting  (cf. \cite{ko}).  There exists a line bundle $ \bl_\zero $  on  $ M_\zero $  whose Chern class is the cohomology class  $ \big[\om_\zero\big] \, $.  Since  $ \om_\zero $  is exact, we have  $  \big[\om_\zero\big] =0  $,  so the bundle  $ \bl_\zero $  is topologically trivial; yet it has interesting geometry given by a connection  $ \nabla $
 with curvature
 $ \om_\zero $.  Let $ C^\infty\big(\bl_\zero\big) $  be the set of all smooth sections of  $ \bl_\zero $. We define the set of
   \hbox{all holomorphic sections by}
  \[  \cH(\bl_\zero)  \; := \;  \{ s \in C^\infty\big(\bl_\zero\big) \;;\;
\nabla_v s = 0  \mbox{ \ for all anti-holomorphic vector fields\ } v \} . \]

The  $ G_\zero $-action on $M_\zero$ lifts to a  $ G_\zero $--representation  on  $ \cH(\bl_\zero) $.  The line bundle has an invariant Hermitian form, namely  $  (s,t) \in C^\infty(M_\zero)  $  for all sections $ s , t \in C^\infty(\bl_\zero) $,  and  $  v(s,t) = (\nabla_v{s}\,,t) + (s\,,\nabla_v{t}) $ for all vector fields  $ v $.

   We now extend this construction to the super setting.  Recall from  \eqref{loc-coord}  that the holomorphic variables are
  $z_p   =   x_p + i\,y_p$ for $p=1,...,n+m$, and $\ze_q  =  \xi_q + i\,\eta_q$
for $q = 1, ...,k$.
Hereafter we use standard multi-index notation, for example
  \[  \ze_{P,Q}   :=   \ze_{p_1} \cdots \ze_{p_r} \, \bar{\ze}_{q_1} \cdots \bar{\ze}_{q_s}   \qquad  \forall \;\; P = (p_1,...,p_r) \, , \; Q=(q_1,...,q_s)  . \]
 The odd variables anti-commute, so we use ascending indices  $  p_1 < \cdots < p_r  $  and  $  q_1 < \cdots < q_s $  so that these
 $ \ze_{P,Q} $  are linearly independent.

   We let the subscript ``top'' denote the presence of all odd variables, so that
  \[  \ze_{\rm top}  =   \ze_1 \cdots \ze_k \, \bar{\ze}_1 \cdots \bar{\ze_k}. \]
 We define the star operator  $ \ze_{P,Q} \mapsto \ze_{P,Q}^*  $,  where  $ \ze_{P,Q}^* $  consists of
%%%%%
% the remaining odd variables,
%%%%%
the odd variables missing from  $ \ze_{P,Q} $,
%%%
 oriented and normalized so that
\beq
  \ze_{P,Q} \, \ze_{P,Q}^*   =   i^{|P|+|Q|} \, \ze_{\rm top} .  \label{abig}
\eeq
 Here  $  i^{|P|+|Q|} \in \{1,i\,\} $,  following  \eqref{ico}.  For  $  Q = \emptyset  $,
 we also write  $ \ze_P = \ze_{p_1} \cdots \ze_{p_r} $  and $  \ze_P \, \ze_P^*  =  i^{|P|} \, \ze_{\rm top} $.
                                                                         \par
   Consider now  $  \sum_{P,Q} f_{P,Q}  \ze_{P,Q} \in
  C^\infty(M) $.
  The Berezin integration  (see \cite[\S4.6]{va})  keeps only the top monomial, namely
\beq
  \int_M \sum\nolimits_{P,Q} f_{P,Q} \, \ze_{P,Q} \, d\bere  \,\; := \;\,  \int_{M_\zero} f_{\rm top} \, dx \, dy \label{berez}
\eeq
 where  $  dx \, dy  $  is the Haar measure of  $ M_\zero $.  The line bundle  $ \bl_\zero $
 extends to a trivial super line bundle  $ \bl $  over  $ M $,  whose smooth sections form the set
  \[
 C^\infty(\bl)
  =  C^\infty(\bl_\zero) \otimes
{\textstyle \bigwedge}_\br ( \zeta_1,\dots,\zeta_k,\bar{\zeta}_1,\dots,\bar{\zeta}_k )  \]
 consisting of linear combinations of  $  s_\zero\,\ze_{P,Q} $  where  $ s_\zero \in C^\infty(\bl_\zero) $.
 We construct an  $ L^2 $-structure  on $ C^\infty(\bl) $
 by
\beq
 \hskip-7pt   \langle s_\zero\,\ze_{P,Q} \, , t_\zero\,\ze_{R,S} \rangle  \; =  \int_M ( s_\zero \, , t_\zero ) \,
\ze_{P,Q} \, \ze_{R,S}^* \, d \bere  \; ,   \!\quad  \forall \;\; s_\zero\,\ze_{P,Q} \, , \, t_\zero\,\ze_{R,S} \, \in \,
 C^\infty (\bl) .
  \label{elt}
\eeq

      The space of holomorphic sections of  $ \bl $  is given by
\beq
  \cH(\bl)  \,\; := \;\,  \big\{\, \sum\nolimits_P s_P\,\ze_P \;;\; s_P \mbox{ is a holomorphic section of } \, \bl_\zero \,\big\} \;. \label{taad}
\eeq
 In  \eqref{taad},  we write  $ \ze_P $  instead of  $ \ze_{P,Q} $  because the anti-holomorphic odd variables do not appear.
%%%%%
% Then we define the set of square integrable holomorphic sections, namely
%%%%%
 The set of square integrable holomorphic sections
   \hbox{is defined by}
  \[  \cH^2(\bl)  \; := \;  \big\{ s \in \cH(\bl) \;;\; \langle s \, , s \rangle \mbox{ converges} \big\} .  \]

\begin{proposition}
 $ \cH^2(\bl) $  is a super Hilbert space.   \label{suhe}
\end{proposition}
\begin{proof}
 Let  $  s = s_\zero\,\ze_{P,Q}  $  and  $ \, t= t_\zero\,\ze_{R,S}  $.  By  \eqref{abig},  \eqref{berez}  and \eqref{elt},  we have
\beq
   \langle s , t \rangle  \,\; = \;\,  i^{|s|} \, \de_{P,R} \, \de_{Q,\,S} \, \int_{M_\zero} \!\!\! ( s_\zero \, , t_\zero ) \, dx \, dy   \label{obee}
\eeq
 where  $ \de_{P,R} $  and  $ \de_{Q,\,S} $  are Kronecker deltas.
 We want to show that this is a super Hermitian metric on the elements that converge, namely it satisfies  \eqref{abi}.

   If  $ s $  and  $ t $ have different parities, then  $ \de_{P,R} \, \de_{Q,\,S} = 0 $,  and so $ \langle s  , t \rangle = 0 $;
   this implies the consistency property  \eqref{abi}(a).
                                                      \par
   Next we check the super Hermitian symmetric property  \eqref{abi}(b).  Here  $ \, \de_{P,R} \, \de_{Q,\,S} = 1 \,$
   only if  $ s $  and  $ t $  have the same parity, and in that case  $ {(-1)}^{|s| \cdot|t|} = {(-1)}^{|s|} $.  Hence
\beq
   {(-1)}^{|s| \cdot|t|} \, \overline{i^{|s|}} \, \de_{P,R} \, \de_{Q,\,S}  \; = \;  i^{|s|} \, \de_{P,R} \, \de_{Q,\,S}  \;.  \label{chib}
\eeq
 By  \eqref{obee}  and  \eqref{chib},  we have
  $$  \begin{array}{rl}
     {(-1)}^{|s| \cdot|t|} \, \overline{\langle t \, , s \rangle}  &
      = \;\,  {(-1)}^{|s| \cdot|t|} \, \overline{i^{|s|}} \, \de_{P,R} \, \de_{Q,\,S} \, \overline{{\textstyle \int_{M_\zero}} ( t_\zero \, , s_\zero ) \, dx \, dy}  \\
     & = \;\,  i^{|s|} \, \de_{P,R} \, \de_{Q,\,S} \, {\textstyle \int_{M_\zero}} ( s_\zero \, , t_\zero ) \, dx \, dy  \,\; = \;\,  \langle s \, , t \rangle  \;.
     \end{array}
       $$
 This proves  \eqref{abi}(b).

   Finally, we check for super positivity  \eqref{abi}(c).  Let  $  s \neq 0  $.  By  \eqref{obee},  if  $ s $  is even (resp. odd),
   then  $ \langle s  , s \rangle \in \br^+  $  (resp. $  \langle s  , s \rangle \in i\,\br^+  $).
   This proves  \eqref{abi}(c).  Thus we have shown that the  $ L^2 $-structure  \eqref{obee}  is indeed a super Hermitian metric on the elements that converge.                                                                  \par
   As ordinary vector spaces, both
  $ {C^\infty(\bl)}_\zero $  and  $ {C^\infty(\bl)}_\one $
  have ordinary inner products  $  \langle \,\cdot\, , \,\cdot\, \rangle $  and  $  -i \, \langle \,\cdot\, , \,\cdot\, \rangle  $  respectively.
  This induces a metric space structure  $ C^\infty(\bl) \, $,  so its completion $ L^2(\bl) $  is a super Hilbert space.  The Bergman space
  $ \cH^2(\bl) = L^2(\bl) \cap \cH(\bl) $  is a closed  subspace of  $ L^2(\bl) $,  so  $ \cH^2(\bl) $  is complete as well,
  and hence it is a super Hilbert space.
\end{proof}

\vskip9pt

   We shall study  $ \cH(\bl) $  and  $ \cH^2(\bl) $  under the  $ G $-action.
   However, it is more convenient to work with functions than sections; to this end, the next proposition provides a global trivialization.
   As before,  $ F $ is the function in Theorem \ref{kah2}.

\vskip11pt

\begin{proposition}
 There exists a nowhere vanishing  $ G $-invariant  section  $ u \in \cH(\bl) $  such that  $ (u,u) = e^{-\F} $.   \label{nonv}
\end{proposition}
\begin{proof}
 In the ordinary setting, there exists a nowhere vanishing  $ G_\zero $-invariant  section  $  u \in \cH(\bl_\zero)  $  such that  $ (u,u) = e^{-\F} $ (see  \cite[(3.13)]{jga}; it has  $ e^{-F} $  because its  $ \om_\zero $  is half of ours).  It extends naturally to a holomorphic section of  $ \bl $,  where  $ u $  is independent of odd variables.  So for all  $ \xi \in \go  $,  we have  $ \frac{\pa}{\partial \xi} u = 0 $,  and this together with the  $ G_\zero $-invariance guarantees that  $ u $  is  $ G $-invariant.
\end{proof}

\vskip9pt

   We consider the set of all holomorphic functions on  $ M $,  namely
  \[  \cH (M)  \; := \;
  \{\, {\textstyle \sum\limits_P} \, f_P \, \ze_P \;;\; f_P \in \cH(M_\zero) \,\; \forall P \,\}  .  \label{smo}  \]
 The section  $ u $  of  Proposition \ref{nonv}  is holomorphic and nowhere vanishing, so it leads to a  $ G $-equivariant  global trivialization
\beq  \label{isom-hol}
  \cH(\bl) \, \cong \, \cH(M) \;\; ,  \quad  f u \mapsto f
\eeq
 that allows us to study  $ \cH^2(\bl) $  in terms of holomorphic functions.

   Extend the star operator of  \eqref{abig}  to
%%%%%%%%%
 $ C^\infty(M) $
%%%%%%%%%
%%%%%%%%%
  by
\beq
   * : C^\infty(M) \lra C^\infty(M) \;\; ,  \quad
   ( f_\zero \, \ze_{P,Q} ) {( f_\zero \, \ze_{P,Q} )}^*  \; =
   \;  i^{|P| + |Q|} \, f_\zero \, \bar{f}_\zero \, \ze_{\rm top}  \;.  \label{starr}
\eeq
 Then define an  $ L^2 $-structure  on
%%%%%%%%%
 $ C^\infty(M) $
%%%%%%%%%
%%%%%%%%%
  by
\beq
  \langle f\,,h \rangle  \; :=
  \;  \int_M f \, h^* \, e^{-\F} \, d \bere  \quad \qquad  \forall \; f , h \in
%%%%%%%%%
 C^\infty(M)
%%%%%%%%%
%%%%%%%%%
\;.  \label{elto}
\eeq
 Using arguments similar to those applied for  Proposition \ref{suhe},
 one sees that this is a super Hermitian metric on the elements which converge, hence we let  $ \, L^2( M , e^{-\F} ) $  denote its completion.  Then we define the Bergman space
 \[  \cH^2( M , e^{-\F})  \; := \;  L^2( M , e^{-\F}) \cap \cH(M) \;.  \]
 Note that we have inserted the weight  $ e^{-\F} $  in  \eqref{elto}  because, by  Proposition \ref{nonv},
 the  $ L^2 $-structures  of  \eqref{elt}  and  \eqref{elto}  are related by
\beq
  \langle f u \, , h u \rangle  \,\; = \;  \int_M f \, h^* (u\,,u) \, d \bere  \,\; = \;  \int_M f \, h^* \, e^{-\F} \, d \bere  \,\; = \;  \langle f , h \rangle   \label{ponx}
\eeq
%%%%%%%%%
% for all $fu, hu \in C^\infty(\bl)$.
%%%%%
 for all  $ \, f , h \in C^\infty(M) \, $.
%%%%%%%%%
%%%%%%%%%
 So we have an isomorphism of  $ G $-modules  and super Hilbert spaces, namely
\beq  \label{isome}
  \cH^2(\bl)  \; \cong \;  \cH^2\big( M , e^{-\F} \,\big)  \;\; ,  \quad  f\,u \mapsto f \;\; .
 \eeq

   We consider now whether the $G$--module $ \cH^2(M , e^{-\F}) $  is unitary, in the sense of
    (\ref{cmpt}).
    The next proposition shows that this is not the case, and characterizes its largest subrepresentation which is unitary.

\vskip7pt

\begin{proposition}  \label{maxu}
 The subspace  $  \cH^2( M_\zero , e^{-\F} )  $  is the largest  $  G $-subrepresentation  of  $  \cH^2(M , e^{-\F} )  $  which is a unitary  $ G $-representation.
\end{proposition}
\begin{proof}
Let $  \xi \in \fg_1  $,  and let $ \frac{\pa}{\pa \xi}  $  denote the left invariant vector field on  $ M $
 associated with $\xi$, here realized as a derivation of functions on  $ M  $.
Suppose that  $ V $  is a  $ G $-subrepresentation of  $  \cH^2(M , e^{-\F} ) $,  and also that
$ \, V \not\subseteq \cH^2(M_\zero,e^{-\F})  $. Then there exists  $ \, f \in V \, $  which is dependent of the odd variables, namely  $\frac{\pa}{\pa \xi} f \neq 0 $  for some  $  \xi \in \fg_1  $.
If the $G$-action on $V$ is unitary, then, by Proposition \ref{shee},  $ \fg_1 $  acts trivially,
which contradicts $ \frac{\pa}{\pa \xi} f \neq 0 $.
We conclude that every subrepresentation of $ \cH^2(M,e^{-\F}) $
 with unitary $G$-action is contained in  $ \cH^2(M_\zero,e^{-\F}) $.

   On the other hand, $\cH^2(M_\zero,e^{-\F})$ is independent of odd variables, so  $\frac{\pa}{\pa \xi} $  acts trivially on it, for all  $ \xi \in \fg_\one  $.
Also, $\cH^2(M_\zero,e^{-\F})$ is clearly a unitary $G_\zero$-representation,
and so it is a unitary $G$-representation  too.
Therefore,  $ \cH^2(M_\zero,e^{-\F}) $  is the largest $ G $--subrepresentation  of  $ \cH^2(M,e^{-\F}) $  in which the  $ G $-action  is unitary.
\end{proof}

\vskip9pt

Let $dr$ be the product of point mass of $\bz^n$ and Lebesgue measure of $\br^m$.
Each element of $\br^m$ has zero measure,
so in the decomposition of unitary $G$-representation into $\wh{G}$-components,
we replace the direct sum by direct integral \cite{ma}.
In the following definition, $\int$ denotes the combination of summation on $\bz^n$ and integration on $\br^m$.

\vskip9pt

\begin{definition}
{\rm Let $\cH$ be a unitary super representation of $G\,$.
We say that  $V_\la^\ep \in \wh{G} \cong (\bz^n \times \br^m) \times \bz_2$  occurs in $\cH$
if there exists  $ f \in \cH $  with parity $\ep$ of the form
\beq
  f(z)  \,\; = \;\,  \int_{\bz^n \times \br^m} \! h(r) \, e^{i r z} \, dr \;,  \label{peac}
\eeq
where  $ \, h(r) \neq 0 \, $  for all $r$ sufficiently near $\la\,$.
If $h(r)$ is almost unique for all $r$ near $\la\,$,
we say that $(\la, \ep)$ occurs with multiplicity one.}  \label{weakocc}
\end{definition}

\vskip9pt

The phrase ``almost unique'' in Definition \ref{weakocc} means that
if we replace $h$ by $h_1$ in (\ref{peac}),
then there is a neighborhood
$U$ of $\la$ such that
$ h(r) = h_1(r)  $  for almost all  $ r \in U $   --- i.e. for all those $r$ outside some subset of measure zero.

\vskip9pt

\begin{example}
{\rm The Fourier transform
expresses every  $  f \in L^2(\br)  $  almost uniquely
as  $  f(z) = \int_\br h(r) \, e^{irz} \, dr  $   --- see for instance \cite[\S7]{ru};
hence every member of $\wh{\br}$ occurs with multiplicity one in $L^2(\br)$.}
\end{example}

\sk

\noindent {\it Proof of Theorem \ref{occ2}:}

Let  $\om$  be a Hamiltonian $G$-invariant super K\"aher form on $M$,
as characterized by Theorem \ref{kah2}. See (\ref{bbzy}).
There exists a line bundle $\bl_\zero$ on $M_\zero$ which corresponds to
$\om_\zero\,$,  cf.\ \cite{ko}. It extends to a super line bundle $\bl$ on $M$,
whose holomorphic sections $\ch(\bl)$ consist of  $ \, \sum_P s_P \, \zeta_P \, $,
where $s_P$ are holomorphic sections of $\bl_\zero $.
 By  Propositions \ref{suhe},  \eqref{isome} and Proposition \ref{maxu},
 $ \cH^2(\bl) \cong \cH^2(M,e^{-\F}) $  are super Hilbert spaces, and
\beq
  \cH^2(\bl_\zero)   \cong   \cH^2(M_\zero,e^{-\F})
  \label{bju}
\eeq
are their respective largest subrepresentations  in which the  $ G $-actions are unitary.

   The elements of  $ \ch^2(M_\zero , e^{-2F} ) $ are independent of odd variables,
   so  $\fg_\one$  acts trivially on it.  We consider the irreducible $G_\zero$-representations
which occur in its direct integral decomposition,
in the sense of Definition \ref{weakocc}.
By \cite[Thm.1.2]{pems},  $ \, V_\la^+ \in \wh{G_\zero} \, $  occurs in  $ \ch^2( M_\zero , e^{-2F}) $
if and only if  $ \, \la \in \im(F') \, $
(\cite[(1.4)]{pems} differs from us by a factor of 2).
By Theorem \ref{kah2}, this is equivalent to  $ \la \in {\im(\Phi)}_\zero $.
Finally, as the elements of  $ \ch^2( M_\zero , e^{-2F}) $  are independent of odd variables,
it does not contain any  $ V_\la^-  $.  This concludes the proof of  Theorem \ref{occ2}.   \hfill  $\Box$

\vskip11pt

For $G_\zero = T_n\,$, $\ch^2(\bl_\zero)$ can be 0 or 1-dimensional,
because  $  {\im(\Phi)}_\zero \cap \bz^n  $  can be $\emptyset$ or $\{\mu\}$.
On the contrary, for  $ G_\zero = \br^m $,  $\ch^2(\bl_\zero) = \int_{{\im(\Phi)}_\zero} V_\la^+ $
cannot be 0 or irreducible because for a strictly convex function $F$,
the image of $F'$ (and hence ${\im(\Phi)}_\zero$) is an open set of $\br^m$.

\vskip9pt

   We extend Gelfand's notion of model of a Lie group \cite{gz},  by saying that a model  of the 
% connected Abelian 
 Lie supergroup  $ G $  is a unitary  $ G $-representation on a super Hilbert space in which every member of  $ \wh{G} $  occurs with multiplicity one.  To construct such a model,  we need a strictly convex function whose gradient mapping is surjective.

\vskip13pt

\noindent {\it Proof of Corollary \ref{mode}:}

Let  $ G = T_n \times \br^m \times {\textstyle \bigwedge}_k^\br $.  Let  $  F \in C^\infty(\br^{n+m})  $
 be a strictly convex function whose gradient mapping  $ F' $  is surjective,
 for instance  $ F(x) = \sum_1^{n+m} x_i^2 $.
 By Theorems \ref{kah2} and \ref{occ2},
we have  $\cH^2(\bl_\zero) \cong  \int_{\bz^n \times \br^m} V_\la^+  $.

   Recall that we have the involutive endofunctor  $ \Pi $  which switches parity
    (see  \eqref{switch}).  So every member of  $ \wh{G} $ occurs exactly once in
  $$  \cH^2(\bl_\zero) \oplus \Pi \, \cH^2(\bl_\zero)  \cong  \int_{\bz^n \times \br^m} V_\la^+ \oplus V_\la^-  $$
 and therefore this is a model of  $ G \, $.  \hfill  $ \Box $

% \newpage
\bigskip
   \bigskip

%%%%%%%%%%%%%%%%%%%%%%%%%%%%%%%%%%%%%%%%%%%%%%%%%%%%%%%%%%%%%%%%%%
%%%%%%%%%%%%%%%%%%%%%%%%%%%%%%%%%%%%%%%%%%%%%%%%%%%%%%%
%%%%%%%%%%%%%%%%%%%%%%%%%%%%%%%%

\section{Beyond irreducibility and unitarity}
\setcounter{equation}{0}
\label{nono}

 By Theorem \ref{occ2}, $ \, \cH^2(\bl_\zero) \, $ is the largest $ G $--subrepresentation of
 $ \cH^2(\bl) \, $  in which the  $ G $-action is unitary, and it decomposes into irreducible subrepresentations
 indexed by the image of the moment map.
 We now address the problem of whether $ \, \cH(\bl) \, $ contains any subrepresentation
 beyond $ \, \cH(\bl_\zero) \, $ which is irreducible, or
is unitarizable with respect to any super metric.
 In view of the trivialization $ \, \cH(\bl) \, \cong \, \cH(M) \, $
 provided by the invariant section of $\bl$,
 we may conduct our discussion on $ \, \cH(M) \, $.

Let us consider the factorization
\beq
\label{eq: tens-split H^2}
  \cH ( M )  \, = \,  \cH ( M_\zero ) \otimes \Lambda_\C^k \;\;,\;\;
 \Lambda_\C^k =  \Lambda_\C(\xi_1\,,\dots,\xi_k)  . \eeq
 Since $ G $  is connected, the $ G $-action  on $ \cH ( M ) $
  is uniquely determined by the $\fg $-action,  which is by super derivations.
   The elements in  $ \cH( M_\zero) $  depend only on even variables, hence are
    annihilated by the super derivations of elements of $\go$.
    In other words,  $ \go $  acts trivially on  $ \cH( M_\zero) $,  and so does  $ G_{\bar{1}} $.
    Similarly, the elements in  $ \Lambda_\C^k $  depend only on odd variables, so they are
     annihilated by the elements of  $ \gz $.
     Thus  $ \gz $  acts trivially on  $\Lambda_\C^k $,  and so does  $ G_{\bar{0}} $.
  This means that, by  \eqref{eq: tens-split H^2},
  $ \cH( M )  $  arises from tensoring the  $ G_{\bar{0}} $-module
   $ \cH( M_\zero) $ and the
    $ G_{\bar{1}}$-module  $ \Lambda_\C^k$.

   The group product in  $ M $  induces a generalized coproduct
  $$  \Delta : C^\infty(M) \relbar\joinrel\lra C^\infty(M \times M) \;\; ,  \quad  f \mapsto \Delta(f) \;\;
  ( (m_1  , m_2) \mapsto f(m_1 \cdot m_2) )   $$
 in the sense of generalized coalgebra theory: indeed, this makes sense because there exists a canonical identification
%%%%%%%%%
  $$  C^\infty(M_{\bar{0}} \times M_{\bar{0}})  \; \cong \;  C^\infty(M_{\bar{0}}) \;
  \widehat{\otimes} \; C^\infty(M_{\bar{0}}) $$
 where the right-hand side is a suitable topological completion of the algebraic tensor product
 $  C^\infty(M_{\bar{0}}) \otimes C^\infty(M_{\bar{0}})  $.
In addition, there exists a dense subalgebra  $ C' \subset C^\infty(M_\zero) $
%%%%%%%%%
%%%%%%%%%
 such that  $ \Delta(C') \subset C' \otimes C' $. Therefore  $ \Delta $  is uniquely determined by its restriction to  $ C' $,
 where for any  $ f \in C' $,
\beq   \label{eq: Delta(f)}
   \Delta(f)   =  {\textstyle \sum\limits_{i=1}^k} \, f'_i \otimes f''_i   \qquad  \text{for suitable \ }  f'_i , f''_i \in C' .
\eeq
See \cite[p.161]{bcc} and references therein.

   Let  $ \xi \in \go $,  and let $ m_{C^\infty(M)} $  be the multiplication in  $ C^\infty(M) $.  The left invariant vector field
   $\cD_\xi$
   is given by  $ \cD_\xi := \, m_{{}_{C^\infty(M)}} \circ ( \id_{{}_{C^\infty(M)}}  \otimes \xi ) \circ \Delta $.
   By  \eqref{eq: Delta(f)},  this means
\[
%\label{eq: Dxi(f)}
   \cD_\xi   =   {\textstyle \sum\limits_{i=1}^k} \, f'_i \cdot \big(\xi.f''_i\big)   \qquad \qquad  \forall \;\; f \in C'
\]
 where  $ \, \xi.f''_i \, $  denotes the scalar obtained by applying $ \xi $ to the germ of function of  $ f''_i $  at  $ e $.
 Since $\cD_\xi$  is a superderivation, its action on any $f$
 is determined by the generators $\xi_1,...,\xi_k$; moreover it suffices to consider $\xi$
 ranging in the $\br$-basis $\{\pa_{\xi_i} := \pa /\pa \xi_i \;;\; i=1,...,k\}$ of $\go$.
 Taking into account that  $ \Delta(\xi_j) = \xi_j \otimes 1 + 1 \otimes \xi_j  $,  we have
\begin{equation}  \label{eq: D.xii(xij)}
  \cD_{\partial_{\xi_i}}\xi_j  \, = \,  \xi_j \cdot \partial_{\xi_i}1 + 1 \cdot \partial_{\xi_i}\xi_j  \, = \,  \delta_{ij}   \qquad \qquad  \forall \;\; i, j = 1 , \dots , k \, .
\end{equation}

   Now recall that  $  \Lambda_\C^k $  is an  $ \N $--graded  algebra, with  $  \big|\xi_j\big| := 1  $  for all  $ j $.  Consider the associated filtration
\begin{equation*}  \label{eq: Lambda-filtration}
   \Lambda_\C^k  \; = \;  \Lambda_\C^{\leq k}  \;\supseteq\;  \Lambda_\C^{\leq k-1}  \;\supseteq\;  \cdots \;\supseteq\;  \Lambda_\C^{\leq 1}  \;\supseteq\;  \Lambda_\C^{\leq 0}  \, = \, \C \cdot 1
\end{equation*}
 where  $ \, \Lambda_\C^{\leq s} := \{ f \in \Lambda_\C^k \,\big|\, |f| \leq s \} \, $  for all  $  s = 0 , 1 , \dots , k \, $.  Then  \eqref{eq: D.xii(xij)}  tells us that
%%%%%
 $  \cD_\xi \big( \Lambda_\C^{\leq s} \big) \,\subseteq\, \Lambda_\C^{\leq s-1} $
%%%%%
 for all  $ s $  and for all  $ \xi \in \go $,  hence in short
%%%%%
\begin{equation}  \label{eq: G- push filtration}
   \go \,.\,\Lambda_\C^{\leq s} \,\subseteq\, \Lambda_\C^{\leq s-1}   \qquad \qquad  \forall \;\; s = 0, 1, \dots, k \, .
\end{equation}
%%%%%

\vskip5pt

  As a category, the Lie supergroups $G$ are equivalent to the super Harish-Chandra pairs
    $(G_\zero, \fg)$; see for instance  \cite{Ga1}, \cite{Ga2}  and references therein. In particular, any superspace  is a  $ G $-module  if and only if it is a  $ (G_\zero,\fg) $-module,  the action of  $ G $  being uniquely determined by that of  $ (G_\zero,\fg) \, $   ---  cf.\ \cite[\S 8.3]{ccf}  for details.

   In the present case, the super Harish-Chandra pair corresponding to the Lie supergroup  $ G_{\bar{1}} $  is  $ ( \{1\} , \go)  $.  Moreover, the action of $\go$ on  $ \Lambda_\C^k $  has been described above. Thus we do know  $ \Lambda_\C^k $  as a  $ G_{\bar{1}} $-representation  space.  In addition, since  $  G = G_{\bar{0}} \times G_{\bar{1}} $  is a direct product, and  $ G_{\bar{0}} $  acts trivially on  $ \Lambda_\C^k $,  the  $ G $-action  on the latter is also known.  We now fix some details about it.
Recall that a representation is completely reducible if it is the direct sum
of irreducible subrepresentations.

\vskip9pt

\begin{proposition}  \label{prop: Lambda-NO-semisimple}
As a $G$ or $ G_{\bar{1}}$-representation,
$ \Lambda_\C^k $ is not completely reducible, so it is not unitary.
Its only irreducible subrepresentation is  $ \C\,1_{\Lambda_\C^k} $.
\end{proposition}
\begin{proof}
We first treat $ \Lambda_\C^k$ as a $ G_{\bar{1}}$-representation.
 Assume there is an isomorphism  $  \Lambda_\C^k \cong \oplus_{i \in I} V_i $  for some family  $ \{ V_i \}_{i \in I} $  of irreducible modules.  By  Theorem \ref{dual2} and its proof, each  $ V_i $  is 1--dimensional.
 But the action of $\go$ switches the parity in  $ V_i  $,  hence such an action is necessarily trivial.
 Likewise, $ G_{\bar{1}}$  acts trivially on each  $ V_i  $.  But this contradicts  \eqref{eq: G- push filtration}. Hence $ \Lambda_\C^k $ is not completely reducible.

   Proposition \ref{prop: unitary => semisimple} says that unitarity implies complete reducibility.  As we have shown that $ \Lambda_\C^k $ is not completely reducible, we conclude that it is not unitary either.

   The non-unitarity of $ \Lambda_\C^k $  also follows from Proposition \ref{shee}:
if $ \Lambda_\C^k $  were unitary, then  Proposition \ref{shee}  implies that $\fg_\one$ acts trivially on it,
which contradicts  \eqref{eq: D.xii(xij)}.

   Finally,  \eqref{eq: D.xii(xij)}  implies that any non-zero subrepresentation of  $ \Lambda_\C^k $  necessarily contains  $  \C\,1_{\Lambda_\C^k}  $;  \,then the latter is the unique irreducible subrepresentation of  $ \Lambda_\C^k $.

                                           The same result is true when we treat    $ \Lambda_\C^k $
                                           as a $G$-representation,
   because  $ G =  G_{\bar{0}} \times  G_{\bar{1}} $ is a direct product of Lie supergroups, and $ G_{\bar{0}}$
   acts trivially on  $ \Lambda_\C^k  $.
\end{proof}

\vskip9pt

\begin{proposition}  \label{finalee}
As a $G$ or $ G_{\bar{1}}$-representation,
 the only unitarizable subrepresentation  of  $  \Lambda_\C^k  $  is  $ \C\,1_{\Lambda_\C^k}  $.
\end{proposition}
\begin{proof}
We first treat $ \Lambda_\C^k$ as a $ G_{\bar{1}}$-representation.
Suppose that $W$ is a non-trivial $ G_{\bar{1}}$-subrepresentation of $  \Lambda_\C^k $.
We apply the notion of $ \mathfrak{u}_B(W) $, introduced in
(\ref{eqx}).
By  Proposition \ref{prop: unitary => semisimple}  and  Proposition \ref{prop: Lambda-NO-semisimple},
there exists no super metric  $ B $  on  $ W $  such that
the  $ \go$-action  on  $ W $  factors through  $ \mathfrak{u}_B(W)  $,
i.e. no such  $ B $  is fixed by the action of $ G_{\bar{1}}$.
The same result is true when we treat    $ \Lambda_\C^k $
                                           as a $G$-representation,
   because the action of $ G_{\bar{0}}$  on  $  \Lambda_\C^k  $  is trivial.
\end{proof}

\sk

\noindent
 \textit{Proof of Theorem \ref{thm: L-NO-unitary}:}

By  \eqref{isome}  we have  $  \cH(\bl) \cong \cH(M)  $,  as well as  $ \cH(\bl_\zero) \cong \cH(M_\zero)$:  therefore, it is enough to prove the claim with  $ M $  replacing  $ \bl $.

   Let  $ W $  be any irreducible  $ \fg $--subrepresentation  of  $  \cH( M ) = \cH(M_\zero) \otimes \Lambda_\C^k $.  By the splitting  $ \fg = \gz \oplus \go $  (cf. \S \ref{Abel-conn-sgroups}),   $ W $  has the form  $ W = W' \otimes W'' $,  with irreducible  subrepresentations  $ W' \subset \cH(M_\zero) $  and  $ W'' \subset \Lambda_\C^k $.  By
   Proposition \ref{prop: Lambda-NO-semisimple},  $  W'' = \C\,1_{\Lambda_\C^k} $,
   so  $  W' \otimes W'' \subset \cH(M_\zero) \otimes \C\,1_{\Lambda_\C^k} = \cH(M_\zero)  $.

   Similarly, if $ W = W' \otimes W''$ is a unitarizable $\fg$--subrepresentation of
   $ \cH( M ) $, then by Proposition \ref{finalee},
   we have $W'' = \C\,1_{\Lambda_\C^k}  $ and hence $W \subset \cH(M_\zero) \otimes \C\,1_{\Lambda_\C^k} = \cH(M_\zero) $.
   This completes the proof of Theorem \ref{thm: L-NO-unitary}.
    \hfill  $ \square $

\bigskip
   \bigskip

\noindent {\bf Data Availability.}

No data sets were generated or analysed
during the current study.

\bigskip

\noindent {\bf Declarations.}

On behalf of all authors, the corresponding author states that there is no conflict of interest.

% \newpage
\bigskip
   \bigskip

 %%%%%%%%%%%%%%%%%%%%%%%%%%%%%%%%%%%%%%%%%%%%%%%%%%%%%%%%%%%%%
%%%%%%%%%%%%%%%%%%%%%%%%%%%%%%%%%%%%%%%%%%%%%%%%%%%%%%%%%%%%%
%%  REFERENCES
%%%%%%%%%%%%%%%%%%%%%%%%%%%%%%%%%%%%%%%%%%%%%%%%%%%%%%%%%%%%%
 %%%%%%%%%%%%%%%%%%%%%%%%%%%%%%%%%%%%%%%%%%%%%%%%%%%%%%%%%%%%%


\begin{thebibliography}{999}

\bibitem{bcc}
  L.\ Balduzzi, C.\ Carmeli, G.\ Cassinelli,  \textit{Super  $ G $--spaces},  in.: D.\ Babbit, V.\ Chari, R.\ Fioresi (eds.),  \textsl{Symmetry in Mathematics and Physics},  Contemp.\ Math.\ \textbf{490}  (2008),  159--176.

\vskip4pt

%
% \bibitem{bcf}
% L. Balduzzi, C. Carmeli, R. Fioresi,  \textit{A comparison of the functors of points of
% supermanifolds},  J.~Algebra Appl.\ \textbf{12} (2013), 1250152.
%
% \vskip5pt
%

\bibitem{ccf}
  C. Carmeli, L. Caston, R. Fioresi,  \textit{Mathematical Foundations of Supersymmetry},
  European Math. Soc. (EMS), Z{\"u}rich, 2011.

\vskip4pt

\bibitem{cw}
S. J. Cheng, W. Wang,
  {\it Dualities and representations of Lie superalgebras},
Grad. Studies in Math., vol. {\bf 144}, Amer. Math. Soc. 2012.

\vskip4pt

\bibitem{jga}
M. K. Chuah,
  {\it K\"ahler structures on complex torus},
J. Geom. Anal. {\bf 10} (2000), 257-267.

\vskip4pt

\bibitem{pems}
M. K. Chuah,
{\it The direct integral of some weighted Bergman spaces},
Proc. Edinburgh Math. Soc. {\bf 50}
(2007), 115-122.

%\vskip4pt

%\bibitem{dm}
 % P. Deligne, J. W. Morgan,  \textit{Notes on supersymmetry (following {J}oseph
% {B}ernstein)}, in:  \textsl{Quantum Fields and Strings: a Course for Mathematicians},
% Vols.~1,~2 (Princeton, NJ, 1996/1997), Amer. Math. Soc., Providence, RI, 1999, 41--97.

\vskip4pt

\bibitem{dw}
  B. DeWitt,  \textit{Supermanifolds},  Cambridge Monographs on Math. Physics, Cambridge University Press, Cambridge, 1984.

\vskip4pt

\bibitem{fg}
  R. Fioresi, F. Gavarini,
  {\it Real forms of complex Lie superalgebras and supergroups},
Commun. Math. Phys.  {\bf 397}  (2023), 937--965.

\vskip4pt

\bibitem{Ga1}
  F. Gavarini,  \textit{Global splittings and super Harish-Chandra pairs for affine supergroups},
Trans. Amer. Math. Soc.  \textbf{368}  (2016), no.\ 6, 3973--4026.

\vskip4pt

\bibitem{Ga2}
  F. Gavarini,  \textit{Lie supergroups vs. super Harish-Chandra pairs: a new equivalence},
Pacific J. Math.  \textbf{306}  (2020), no. 2, 451--485.

\vskip4pt

\bibitem{gz}
I. M. Gelfand, A. Zelevinski,
{\it Models of representations of classical groups and their
hidden symmetries},
Funct. Anal. Appl. {\bf 18} (1984), 183-198.

\vskip4pt

\bibitem{gs}
V. Guillemin, S. Sternberg,
  {\it Symplectic techniques in physics},
Cambridge Univ. Press, 1984.

\vskip4pt

\bibitem{ko}
B. Kostant,
  {\it Quantization and unitary representations},
Lecture Notes in Math. {\bf 170}, pp.87-208,
Springer-Verlag, New York/Berlin 1970.

\vskip4pt

\bibitem{ma}
G. Mackey,
{\it The theory of unitary group representations},
Univ. Chicago Press 1976.

\vskip4pt

\bibitem{ns}
K.-H. Neeb and H. Salmasian,
{\it Lie supergroups, unitary representations, and invariant cones},
in ``Supersymmetry in Mathematics and Physics'',
Lecture Notes in Math. {\bf 2027}, pp.195-239,
Springer 2011.


\vskip4pt

\bibitem{ro}
  A. Rogers,  \textit{Supermanifolds. Theory and applications},
  World Scientific Publishing Co. Pte. Ltd., Hackensack, NJ, 2007.

\vskip4pt

\bibitem{ru}
W. Rudin,
{\it Functional analysis}, McGraw-Hill, Columbus OH 1973.

\vskip4pt

\bibitem{sa}
H. Salmasian, {\it Unitary representations of nilpotent super Lie groups},
Comm. Math. Phys. {\bf 297} (2010), 189-227.

%
%\vskip4pt
%
% G. M. Tuynman,  \textit{Supermanifolds and supergroups. Basic theory},
% Math. and its Applications  \textbf{570},  Kluwer Academic Publishers, Dordrecht, 2004.
%
% \vskip4pt
%
% G. M. Tuynman,
%  {\it Super symplectic geometry and prequantization},
% J. Geom. Phys. {\bf 60} (2010), 1919-1939.

\vskip4pt

\bibitem{va}
V. S. Varadarajan,
  {\it Supersymmetry for mathematicians: an introduction},
Courant Lecture Notes {\bf 11},
Amer. Math. Soc. 2004.

\vskip4pt

\bibitem{vars}
S. Varsaie, {\it K\"ahlerian supermanifolds},
J. Math. Phys. {\bf 40} (1999), no. 12, 6701-6706.

%\vskip4pt

%\bibitem{Ya1}
 % Y.\ F.\ Yao, {\it Non-restricted representations of simple Lie superalgebras of special type and Hamiltonian type},
  % Sci.\ China Math.  {\bf 56}  (2013), no.\ 2, 239--252.

%\vskip4pt

%\bibitem{Ya2}
 %Y.\ F.\ Yao, {\it The abelian subalgebras of maximal dimensions for general linear Lie superalgebras},
  %Linear and Multilinear Algebra  {\bf 64}  (2016), no.\ 10, 2081--2089.

\end{thebibliography}
\end{document}